\documentclass[11pt,reqno,tbtags]{amsart}
\usepackage{amssymb}
\usepackage{hyperref}
\usepackage{url}
\usepackage[square,numbers]{natbib}
\bibpunct[, ]{[}{]}{;}{n}{,}{,}

\title[Pareto Record Frontier:\ Localization of the Trailing Point]
{Sharpened Localization of the Trailing Point of the Pareto Record Frontier}

\newcommand\urladdrx[1]{{\urladdr{\def~{{\tiny$\sim$}}#1}}}
\author{James Allen Fill}
\address{Department of Applied Mathematics and Statistics,
The Johns Hopkins University,
3400 N.~Charles Street,
Baltimore, MD 21218-2682 USA}
\email{jimfill@jhu.edu}
\urladdrx{http://www.ams.jhu.edu/~fill/}
\thanks{Research for all authors supported by
the Acheson~J.~Duncan Fund for the Advancement of Research in
Statistics.}

\author{Daniel~Q.\ Naiman}
\address{Department of Applied Mathematics and Statistics,
The Johns Hopkins University,
3400 N.~Charles Street,
Baltimore, MD 21218-2682 USA}
\email{daniel.naiman@jhu.edu}
\urladdrx{https://www.ams.jhu.edu/~dan/}

\author{Ao Sun}
\address{Department of Applied Mathematics and Statistics, 
The Johns Hopkins University,
3400 N.~Charles Street,
Baltimore, MD 21218-2682 USA}
\email{asun17@jhu.edu}

\makeatletter
\@namedef{subjclassname@2020}{\textup{2020} Mathematics Subject Classification}
\makeatother

\keywords{Multivariate records, Pareto records, generators, interior generators, minima, maxima, record-setting region, frontier, current records, boundary-crossing probabilities, first moment method, second moment method, orthants}
\subjclass[2020]{Primary:\ 60D05; Secondary:\ 60F05, 60F15, 60G70, 60G17}
\numberwithin{equation}{section}

\renewcommand\le{\leqslant}
\renewcommand\ge{\geqslant}

\allowdisplaybreaks

\theoremstyle{plain}
\newtheorem{theorem}{Theorem}[section]
\newtheorem{lemma}[theorem]{Lemma}
\newtheorem{proposition}[theorem]{Proposition}
\newtheorem{corollary}[theorem]{Corollary}

\theoremstyle{definition}
\newtheorem{example}[theorem]{Example}
\newtheorem{definition}[theorem]{Definition}

\newtheorem{remark}[theorem]{Remark}

\theoremstyle{remark}

\newenvironment{romenumerate}[1][-10pt]{
\addtolength{\leftmargini}{#1}\begin{enumerate}
 \renewcommand{\labelenumi}{\textup{(\roman{enumi})}}%
 \renewcommand{\theenumi}{\textup{(\roman{enumi})}}%
 }{\end{enumerate}}

\newcounter{oldenumi}
{\setcounter{oldenumi}{\value{enumi}}
\begin{romenumerate} \setcounter{enumi}{\value{oldenumi}}}
{\end{romenumerate}}

\newcounter{thmenumerate}

\newcounter{xenumerate}   

\newcommand\xfootnote[1]{\unskip\footnote{#1}$ $} 

\newcommand\pfitem[1]{\par(#1):}
\newcommand\pfitemx[1]{\par#1:}
\newcommand\pfitemref[1]{\pfitemx{\ref{#1}}}

\newcommand\pfcase[2]{\smallskip\noindent\emph{Case #1: #2} \noindent}
\newcommand\step[2]{\smallskip\noindent\emph{Step #1: #2} \noindent}
\newcommand\stepx{\smallskip\noindent\refstepcounter{steps}%
 \emph{Step \arabic{steps}:}\noindent}

\newcommand{\refT}[1]{Theorem~\ref{#1}}
\newcommand{\refC}[1]{Corollary~\ref{#1}}
\newcommand{\refL}[1]{Lemma~\ref{#1}}
\newcommand{\refR}[1]{Remark~\ref{#1}}
\newcommand{\refS}[1]{Section~\ref{#1}}
\newcommand{\refSS}[1]{Subsection~\ref{#1}}
\newcommand{\refP}[1]{Proposition~\ref{#1}}
\newcommand{\refD}[1]{Definition~\ref{#1}}
\newcommand{\refE}[1]{Example~\ref{#1}}
\newcommand{\refF}[1]{Figure~\ref{#1}}
\newcommand{\refApp}[1]{Appendix~\ref{#1}}
\newcommand{\refTab}[1]{Table~\ref{#1}}
\newcommand{\refand}[2]{\ref{#1} and~\ref{#2}}

\newcommand\marginal[1]{\marginpar{\raggedright\parindent=0pt\tiny #1}}
\newcommand\DN{\marginal{DN}}
\newcommand\JF{\marginal{JF 2023.09.30}}

\newcommand\kolla{\marginal{CHECK! SJ}}
\newcommand\ms[1]{\texttt{[ms #1]}}
\newcommand\XXX{XXX \marginal{XXX}}
\newcommand\REM[1]{{\raggedright\texttt{[#1]}\par\marginal{XXX}}}
\newcommand\rem[1]{{\texttt{[#1]}\marginal{XXX}}}

\newcommand\linebreakx{\unskip\marginal{$\backslash$linebreak}\linebreak}

\begingroup
  \divide\count255 by 60
  \multiply\count255 by -60
  \advance\count255 by \time
  \ifnum \count255 < 10 \xdef\klockan{\the\count1.0\the\count255}
  \else\xdef\klockan{\the\count1.\the\count255}\fi
\endgroup

\newcommand\nopf{\qed}   
\newcommand\noqed{\renewcommand{\qed}{}} 
\newcommand\qedtag{\eqno{\qed}}

\DeclareMathOperator*{\sumx}{\sum\nolimits^{*}}
\DeclareMathOperator*{\sumxx}{\sum\nolimits^{**}}
\newcommand{\sumio}{\sum_{i=0}^\infty}
\newcommand{\sumjo}{\sum_{j=0}^\infty}
\newcommand{\sumj}{\sum_{j=1}^\infty}
\newcommand{\sumko}{\sum_{k=0}^\infty}
\newcommand{\sumk}{\sum_{k=1}^\infty}
\newcommand{\summo}{\sum_{m=0}^\infty}
\newcommand{\sumno}{\sum_{n=0}^\infty}
\newcommand{\sumn}{\sum_{n=1}^\infty}
\newcommand{\sumin}{\sum_{i=1}^n}
\newcommand{\sumjn}{\sum_{j=1}^n}
\newcommand{\sumkn}{\sum_{k=1}^n}
\newcommand{\prodin}{\prod_{i=1}^n}
\newcommand{\sprod}{\mbox{$\prod$}}
\newcommand\set[1]{\ensuremath{\{#1\}}}
\newcommand\bigset[1]{\ensuremath{\bigl\{#1\bigr\}}}
\newcommand\Bigset[1]{\ensuremath{\Bigl\{#1\Bigr\}}}
\newcommand\biggset[1]{\ensuremath{\biggl\{#1\biggr\}}}
\newcommand\lrset[1]{\ensuremath{\left\{#1\right\}}}
\newcommand\xpar[1]{(#1)}
\newcommand\bigpar[1]{\bigl(#1\bigr)}
\newcommand\Bigpar[1]{\Bigl(#1\Bigr)}
\newcommand\biggpar[1]{\biggl(#1\biggr)}
\newcommand\lrpar[1]{\left(#1\right)}
\newcommand\bigsqpar[1]{\bigl[#1\bigr]}
\newcommand\Bigsqpar[1]{\Bigl[#1\Bigr]}
\newcommand\biggsqpar[1]{\biggl[#1\biggr]}
\newcommand\lrsqpar[1]{\left[#1\right]}
\newcommand\xcpar[1]{\{#1\}}
\newcommand\bigcpar[1]{\bigl\{#1\bigr\}}
\newcommand\Bigcpar[1]{\Bigl\{#1\Bigr\}}
\newcommand\biggcpar[1]{\biggl\{#1\biggr\}}
\newcommand\lrcpar[1]{\left\{#1\right\}}
\newcommand\abs[1]{|#1|}
\newcommand\bigabs[1]{\bigl|#1\bigr|}
\newcommand\Bigabs[1]{\Bigl|#1\Bigr|}
\newcommand\biggabs[1]{\biggl|#1\biggr|}
\newcommand\lrabs[1]{\left|#1\right|}
\def\rompar(#1){\textup(#1\textup)}    
\newcommand\xfrac[2]{#1/#2}
\newcommand\xpfrac[2]{(#1)/#2}
\newcommand\xqfrac[2]{#1/(#2)}
\newcommand\xpqfrac[2]{(#1)/(#2)}
\newcommand\parfrac[2]{\lrpar{\frac{#1}{#2}}}
\newcommand\bigparfrac[2]{\bigpar{\frac{#1}{#2}}}
\newcommand\Bigparfrac[2]{\Bigpar{\frac{#1}{#2}}}
\newcommand\biggparfrac[2]{\biggpar{\frac{#1}{#2}}}
\newcommand\xparfrac[2]{\xpar{\xfrac{#1}{#2}}}
\newcommand\innprod[1]{\langle#1\rangle}
\newcommand\expbig[1]{\exp\bigl(#1\bigr)}
\newcommand\expBig[1]{\exp\Bigl(#1\Bigr)}
\newcommand\explr[1]{\exp\left(#1\right)}
\newcommand\expQ[1]{e^{#1}}
\def\xexp(#1){e^{#1}}
\newcommand\ceil[1]{\lceil#1\rceil}
\newcommand\floor[1]{\lfloor#1\rfloor}
\newcommand\lrfloor[1]{\left\lfloor#1\right\rfloor}
\newcommand\frax[1]{\{#1\}}
\newcommand\setn{\set{1,\dots,n}}
\newcommand\nn{[n]}
\newcommand\ntoo{\ensuremath{{n\to\infty}}}
\newcommand\Ntoo{\ensuremath{{N\to\infty}}}
\newcommand\asntoo{\text{as }\ntoo}
\newcommand\ktoo{\ensuremath{{k\to\infty}}}
\newcommand\mtoo{\ensuremath{{m\to\infty}}}
\newcommand\stoo{\ensuremath{{s\to\infty}}}
\newcommand\ttoo{\ensuremath{{t\to\infty}}}
\newcommand\xtoo{\ensuremath{{x\to\infty}}}
\newcommand\bmin{\wedge}
\newcommand\norm[1]{\|#1\|}
\newcommand\bignorm[1]{\bigl\|#1\bigr\|}
\newcommand\Bignorm[1]{\Bigl\|#1\Bigr\|}
\newcommand\downto{\searrow}
\newcommand\upto{\nearrow}
\newcommand\half{\tfrac12}
\newcommand\thalf{\tfrac12}

\newcommand\punkt{.\spacefactor=1000}    
\newcommand\iid{i.i.d\punkt}
\newcommand\ie{i.e\punkt}
\newcommand\eg{e.g\punkt}
\newcommand\viz{viz\punkt}
\newcommand\cf{cf\punkt}
\newcommand{\as}{a.s\punkt}
\newcommand{\aex}{a.e\punkt}
\newcommand{\io}{i.o\punkt}
\renewcommand{\ae}{\vu}  
\newcommand\whp{w.h.p\punkt}
\renewcommand{\aa}{a.a\punkt}

\newcommand\ii{\mathrm{i}}

\newcommand{\tend}{\longrightarrow}
\newcommand\dto{\overset{\mathrm{d}}{\tend}}
\newcommand\pto{\overset{\mathrm{p}}{\tend}}
\newcommand\Pto{\overset{\mathrm{P}}{\tend}}
\newcommand\Lcto{\overset{\mathcal{L}}{\tend}}
\newcommand\asto{\overset{\mathrm{a.s.}}{\tend}}
\newcommand\eqd{\overset{\mathrm{d}}{=}}
\newcommand\neqd{\overset{\mathrm{d}}{\neq}}
\newcommand\op{o_{\mathrm p}}
\newcommand\Op{O_{\mathrm p}}

\newcommand\bbR{\mathbb R}
\newcommand\bbC{\mathbb C}
\newcommand\bbN{\mathbb N}
\newcommand\bbT{\mathbb T}
\newcommand\bbQ{\mathbb Q}
\newcommand\bbZ{\mathbb Z}
\newcommand\bbZleo{\mathbb Z_{\le0}}
\newcommand\bbZgeo{\mathbb Z_{\ge0}}

\newcounter{CC}
\newcommand{\CC}{\stepcounter{CC}\CCx} 
\newcommand{\CCx}{C_{\arabic{CC}}}     
\newcommand{\CCdef}[1]{\xdef#1{\CCx}}     
\newcommand{\CCname}[1]{\CC\CCdef{#1}}    
\newcommand{\CCreset}{\setcounter{CC}0} 
\newcounter{cc}
\newcommand{\cc}{\stepcounter{cc}\ccx} 
\newcommand{\ccx}{c_{\arabic{cc}}}     
\newcommand{\ccdef}[1]{\xdef#1{\ccx}}     
\newcommand{\ccname}[1]{\cc\ccdef{#1}}    
\newcommand{\ccreset}{\setcounter{cc}0} 

\renewcommand\Re{\operatorname{Re}}
\renewcommand\Im{\operatorname{Im}}

\newcommand\E{\operatorname{\mathbb E{}}}
\renewcommand\P{\operatorname{\mathbb P{}}}
\renewcommand\L{\operatorname{L}}
\newcommand\Var{\operatorname{Var}}
\newcommand\Cov{\operatorname{Cov}}
\newcommand\Corr{\operatorname{Corr}}
\newcommand\Exp{\operatorname{Exp}}
\newcommand\Po{\operatorname{Po}}
\newcommand\Bi{\operatorname{Bi}}
\newcommand\Bin{\operatorname{Bin}}
\newcommand\Be{\operatorname{Be}}
\newcommand\Ge{\operatorname{Ge}}
\newcommand\NBi{\operatorname{NegBin}}
\newcommand\Res{\operatorname{Res}}
\newcommand\fall[1]{^{\underline{#1}}}
\newcommand\rise[1]{^{\overline{#1}}}

\newcommand\supp{\operatorname{supp}}
\newcommand\sgn{\operatorname{sgn}}
\newcommand\Tr{\operatorname{Tr}}

\newcommand\ga{\alpha}
\newcommand\gb{\beta}
\newcommand\tgb{\tilde{\gb}}
\newcommand\gd{\delta}
\newcommand\gD{\Delta}
\newcommand\gf{\varphi}
\newcommand\gam{\gamma}
\newcommand\gG{\Gamma}
\newcommand\gk{\varkappa}
\newcommand\gl{\lambda}
\newcommand\gL{\Lambda}
\newcommand\go{\omega}
\newcommand\gO{\Omega}
\newcommand\gs{\sigma}
\newcommand\gss{\sigma^2}
\newcommand\gth{\theta}
\newcommand\eps{\varepsilon}
\newcommand\ep{\varepsilon}

\newcommand\bJ{\bar J}

\newcommand\cA{\mathcal A}
\newcommand\cB{\mathcal B}
\newcommand\cC{\mathcal C}
\newcommand\cD{\mathcal D}
\newcommand\cE{\mathcal E}
\newcommand\cF{\mathcal F}
\newcommand\cG{\mathcal G}
\newcommand\cH{\mathcal H}
\newcommand\cI{\mathcal I}
\newcommand\cJ{\mathcal J}
\newcommand\cK{\mathcal K}
\newcommand\cL{{\mathcal L}}
\newcommand\cM{\mathcal M}
\newcommand\cN{\mathcal N}
\newcommand\cO{\mathcal O}
\newcommand\cP{\mathcal P}
\newcommand\cQ{\mathcal Q}
\newcommand\cR{{\mathcal R}}
\newcommand\cS{{\mathcal S}}
\newcommand\cT{{\mathcal T}}
\newcommand\cU{{\mathcal U}}
\newcommand\cV{\mathcal V}
\newcommand\cW{\mathcal W}
\newcommand\cX{{\mathcal X}}
\newcommand\cY{{\mathcal Y}}
\newcommand\cZ{{\mathcal Z}}

\newcommand\ett[1]{\boldsymbol1_{#1}}
\newcommand\bigett[1]{\boldsymbol1\bigcpar{#1}}
\newcommand\Bigett[1]{\boldsymbol1\Bigcpar{#1}}
\newcommand\etta{\boldsymbol1}

\newcommand\smatrixx[1]{\left(\begin{smallmatrix}#1\end{smallmatrix}\right)}

\newcommand\limn{\lim_{n\to\infty}}
\newcommand\limN{\lim_{N\to\infty}}

\newcommand\qw{^{-1}}
\newcommand\qww{^{-2}}
\newcommand\qq{^{1/2}}
\newcommand\qqw{^{-1/2}}
\newcommand\qqq{^{1/3}}
\newcommand\qqqb{^{2/3}}
\newcommand\qqqw{^{-1/3}}
\newcommand\qqqbw{^{-2/3}}
\newcommand\qqqq{^{1/4}}
\newcommand\qqqqc{^{3/4}}
\newcommand\qqqqw{^{-1/4}}
\newcommand\qqqqcw{^{-3/4}}

\newcommand\intii{\int_{-1}^1}
\newcommand\intoi{\int_0^1}
\newcommand\intoo{\int_0^\infty}
\newcommand\intoooo{\int_{-\infty}^\infty}
\newcommand\oi{[0,1]}
\newcommand\ooo{[0,\infty)}
\newcommand\ooox{[0,\infty]}
\newcommand\oooo{(-\infty,\infty)}
\newcommand\setoi{\set{0,1}}

\newcommand\dtv{d_{\mathrm{TV}}}

\newcommand\dd{\,\mathrm{d}}
\newcommand\ddx{\mathrm{d}}

\newcommand{\pgf}{probability generating function}
\newcommand{\mgf}{moment generating function}
\newcommand{\chf}{characteristic function}
\newcommand{\ui}{uniformly integrable}
\newcommand\rv{random variable}
\newcommand\lhs{left-hand side}
\newcommand\rhs{right-hand side}

\newcommand\gnp{\ensuremath{G(n,p)}}
\newcommand\gnm{\ensuremath{G(n,m)}}
\newcommand\gnd{\ensuremath{G(n,d)}}
\newcommand\gnx[1]{\ensuremath{G(n,#1)}}

\newcommand\etto{\bigpar{1+o(1)}}
\newcommand\GW{Galton--Watson}
\newcommand\GWt{\GW{} tree}
\newcommand\cGWt{conditioned \GW{} tree}
\newcommand\GWp{\GW{} process}
\newcommand\tB{{\widetilde F}}
\newcommand\tC{{\widetilde C}}
\newcommand\tG{{\widetilde G}}
\newcommand\tI{{\widetilde I}}
\newcommand\tK{{\widetilde K}}
\newcommand\tW{{\widetilde W}}
\newcommand\tX{{\widetilde X}}
\newcommand\tY{{\widetilde Y}}
\newcommand\kk{\varkappa}
\newcommand\spann[1]{\operatorname{span}(#1)}
\newcommand\tn{\cT_n}
\newcommand\tnv{\cT_{n,v}}
\newcommand\rea{\Re\ga}
\newcommand\wgay{{-\ga-\frac12}}
\newcommand\qgay{{\ga+\frac12}}
\newcommand\ex{\mathbf e}
\newcommand\uu{\mathbf u}
\newcommand\vv{\mathbf v}
\newcommand\xx{\mathbf x}
\newcommand\yy{\mathbf y}
\newcommand\zz{\mathbf z}

\newcommand\hF{\widehat F}
\newcommand\hK{\widehat K}
\newcommand\hX{\widehat X}

\newcommand\Bh{B}
\newcommand\sgt{simply generated tree}
\newcommand\sgrt{simply generated random tree}
\newcommand\hh[1]{d(#1)}
\newcommand\WW{\widehat W}
\newcommand\coi{C\oi}
\newcommand\out{\gd^+}
\newcommand\zne{Z_{n,\eps}}
\newcommand\ze{Z_{\eps}}
\newcommand\gatoo{\ensuremath{\ga\to\infty}}
\newcommand\rtoo{\ensuremath{r\to\infty}}
\newcommand\Yoo{Y_\infty}
\newcommand\bes{R}
\newcommand\tex{\tilde{\ex}}
\newcommand\tbes{\tilde{\bes}}
\newcommand\Woo{W_\infty}
\newcommand{\hm}{m_1}
\newcommand{\thm}{\tilde m_1}
\newcommand{\bbb}{B^{(3)}}
\newcommand{\rr}{r^{1/2}}
\newcommand\coo{C[0,\infty)}
\newcommand\coT{\ensuremath{C[0,T]}}
\newcommand\expx[1]{e^{#1}}
\newcommand\gdtau{\gD\tau}
\newcommand\ygam{Y_{(\gam)}}
\newcommand\EE{V}
\newcommand\pigsqq{\sqrt{2\pi\gss}}
\newcommand\pigsqqw{\frac{1}{\sqrt{2\pi\gss}}}
\newcommand\gapigsqqw{\frac{(\ga-\frac12)\qw}{\sqrt{2\pi\gss}}}
\newcommand\gdd{\frac{\gd}{2}}

\newcommand\raisetagbase{\raisetag{\baselineskip}}

\newcommand\eit{e^{\ii t}}
\newcommand\emit{e^{-\ii t}}
\newcommand\tgf{\tilde\gf}
\newcommand\txi{\tilde\xi}

\newcommand\intT{\frac{1}{2\pi}\int_{-\pi}^\pi}
\newcommand\intpi{\int_{-\pi}^\pi}
\newcommand\Li{\operatorname{Li}}
\newcommand\doi{D_{01}}
\newcommand\cHoi{\cH(\doi)}
\newcommand\db{D}
\newcommand\dbm{D_-}
\newcommand\dbmb{\overline D_-}
\newcommand\dbmx{\widehat D_-}
\newcommand\bdb{\overline{D_B}}
\newcommand\xq{\setminus\set{\frac12}}
\newcommand\xo{\setminus\set{0}}
\newcommand\tqn{t/\sqrt n}
\newcommand\intpm[1]{\int_{-#1}^{#1}}
\newcommand\gnaxt{g_n(\ga,x,t)}
\newcommand\gaxt{g(\ga,x,t)}
\newcommand\gssx{\frac{\gss}2}

\newcommand\tq{\tilde q}
\newcommand\tr{\tilde r}
\newcommand\gao{\ga_0}
\newcommand\ppp{\cP_1}
\newcommand\dx{D^*}

\newcommand\tpsi{\tilde\psi}
\newcommand\tgD{\tilde\gD}
\newcommand\xinn{\xi_{n-1,N}}
\newcommand\zzn{\frac12+\ii y_n}
\newcommand\OHD{O_{\cH(D)}}
\newcommand\OHDx{O_{\cH(\dx)}}
\newcommand\tgdn{\tgD_N}
\newcommand\xgdn{\gD^*_N}
\newcommand\intt{\int_0^T}
\newcommand\act{|\cT|}

\newcommand{\ignore}[1]{}

\newcommand{\Holder}{H\"older}
\newcommand\CS{Cauchy--Schwarz}
\newcommand\CSineq{\CS{} inequality}
\newcommand{\Levy}{L\'evy}
\newcommand{\Takacs}{Tak\'acs}
\newcommand{\Frechet}{Fr\'echet}

\newcommand{\maple}{\texttt{Maple}}

\newcommand\citex{\REM}
\newcommand\refx[1]{\texttt{[#1]}}
\newcommand\xref[1]{\texttt{(#1)}}

\usepackage{tikz}
\usepackage{amssymb}
\usetikzlibrary{arrows,positioning}
\usepackage[english]{babel}
\usepackage{pgfplotstable}
\usepackage{pgfplots}
\usepackage{adjustbox}
\pgfplotsset{compat=1.3}

\begin{document}

\date{February~20, 2024}

\maketitle

\begin{abstract}
\begin{scriptsize}
For $d \geq 2$ and \iid\ $d$-dimensional observations $X^{(1)}, X^{(2)}, \ldots$ with independent Exponential$(1)$ coordinates, we revisit the study by Fill and Naiman (\emph{Electron.\ J.\ Probab.},\ 25:Paper No.\ 92, 24 pp.,\ 2020)  of the boundary (relative to the closed positive orthant), or ``frontier'', $F_n$ of the closed Pareto record-setting (RS) region
\[
\mbox{RS}_n
:= \{0 \leq x \in {\mathbb R}^d: x \not\prec X^{(i)}\mbox{\ for all $1 \leq i \leq n$}\}
\]
at time $n$, where
$0 \leq x$ means that $0 \leq x_j$ for $1 \leq j \leq d$ and
$x \prec y$ means that $x_j < y_j$ for $1 \leq j \leq d$.  With $x_+ := \sum_{j = 1}^d x_j$, let
\[
F_n^- := \min\{x_+: x \in F_n\} \quad \mbox{and} \quad
F_n^+ := \max\{x_+: x \in F_n\}.
\]
Almost surely, there are for each~$n$ unique vectors $\gl_n \in F_n$ and $\tau_n \in F_n$ such that $F_n^+ = (\gl_n)_+$ and $F_n^- = (\tau_n)_+$; we refer to $\gl_n$ and $\tau_n$ as the \emph{leading} and \emph{trailing} points, respectively, of the frontier.
Fill and Naiman provided rather sharp information about the typical and almost sure behavior of $F^+$, 
but somewhat crude information about $F^-$, namely, that for any $\eps > 0$ and $c_n \to \infty$ we have 
\[
\P(F_n^- - \ln n \in (- (2 + \eps) \ln \ln \ln n, c_n)) \to 1
\] 
(describing typical behavior) and almost surely 
\[
\limsup \frac{F_n^- - \ln n}{\ln \ln n} \leq 0\mbox{\ and\ }\liminf \frac{F_n^- - \ln n}{\ln \ln \ln n} \in [-2, -1].
\]

In this paper we use the theory of \emph{generators} (minima of $F_n$) together with the first- and 
second-moment methods to improve considerably the trailing-point location results to
\[
F_n^- - (\ln n - \ln \ln \ln n) \Pto - \ln(d - 1)
\]
(describing typical behavior) and, for $d \geq 3$, almost surely
\begin{align*}
&\limsup [F_n^- - (\ln n - \ln \ln \ln n)] \leq -\ln(d - 2) + \ln 2 \\
\mbox{and\ }&\liminf [F_n^- - (\ln n - \ln \ln \ln n)] \geq - \ln d - \ln 2.
\end{align*}
\end{scriptsize}
\end{abstract}

\section{Introduction, background, and main results}\label{S:intro}

{\bf Notation:\ }Throughout 
this paper we abbreviate the $k$th iterate of natural logarithm $\ln$ by $\L_k$ and $\L_1$ by $\L$, and we write $x_+ := \sum_{j = 1}^d x_j$ and $x_{\times} := \prod_{j = 1}^d x_j$ for the sum and product, respectively, of coordinates of the $d$-dimensional vector $x = (x_1, \dots, x_d)$.

Unless otherwise specifically noted, all the results of this paper hold for any dimension $d \geq 2$.

The study of univariate records is well established (\cite{Arnold(1998)} is a standard reference), but that of multivariate records remains 
under vigorous development. 
Fill and Naiman~\cite{Fillboundary(2020)} studied the stochastic process $(F_n)$, where $F_n$ is the boundary, or ``frontier'', for \emph{Pareto records}
(consult Definitions~\ref{D:record}--\ref{D:RS}) in general dimension~$d$ when the observed sequence of points $X^{(1)}, X^{(2)}, \dots$ are assumed (as they are throughout this paper, except 
where otherwise noted) to be \iid\ (independent and identically distributed) copies of a $d$-dimensional random vector~$X$ with independent Exponential$(1)$ coordinates $X_j$.  Their main goal was to sharpen (in various senses) the assertion in Bai et al.~\cite{Bai(2005)} ``that nearly all maxima occur in a thin strip sandwiched between [the] two parallel hyper-planes'' 
\[
x_+ = \L n - \L_3 n - \L[4 (d - 1)] \quad \mbox{and} \quad x_+ = \L n + 4 (d - 1) \L_2 n.
\]
They did this largely by studying (separately) the maximum and minimum sums of coordinates for points lying in $F_n$.  The results for the maximum sum were rather sharp; less so for the minimum sum.  The main aim of this paper is to use the theory of generators (minima of $F_n$) and the first- and second-moment methods to improve considerably their results about the minimum sum.

\subsection{Pareto records and the record-setting region}
\label{S:records}

For the reader's convenience, and with the permission of the authors and the copyright holder, this short subsection is excerpted largely verbatim from \cite[Section~1.1]{Fillboundary(2020)}.
 
We begin with some definitions.  
For a positive integer~$n$, let $[n] := \{1, \dots, n\}$.  Thus $[d]^{[n]}$ denotes the set of all functions
from~$[n]$ into~$[d]$, or simply the set of all $n$-tuples with each entry in $\{1, \dots, d\}$.
For $d$-dimensional vectors $x = (x_1, \dots, x_d)$ and $y = (y_1, \dots, y_d)$,
write $x \prec y$ (respectively, $x \leq y$) to mean that $x_j < y_j$ (resp.,\ $x_j \leq y_j$)
for $j \in [d]$.
(We caution that, with this convention,  $\leq$ is weaker than $\preceq$, the latter meaning ``$\prec$ or $=$''; indeed, $(0, 0) \leq (0, 1)$ but we have neither $(0, 0) \prec (0, 1)$ nor $(0, 0) = (0, 1)$.  
This distinction will 
be important 
for some of our later discussion of generators.)
The notation $x \succ y$ means $y \prec x$, and $x \geq y$ means $y \leq x$; the notation $x < y$ means
$x \leq y$ but $x \neq y$, and $y > x$ means $x < y$.  We write
$x_+ := \sum_{j = 1}^d x_j$ (respectively, $x_{\times} := \prod_{j = 1}^d x_j$) for
the sum (resp.,\ product) of coordinates of $x = (x_1, \dots, x_d)$.

\begin{definition}
\label{D:record}
(a)~We say that $X^{(k)}$ is a \emph{(Pareto) record} (or that it \emph{sets} a record at time~$k$) if $X^{(k)} \not\prec X^{(i)}$ for all $1 \leq i < k$.

(b)~If $1 \leq k \leq n$,
we say that $X^{(k)}$ is a \emph{current record} (or \emph{remaining record}, or \emph{maximum}) at time~$n$ if $X^{(k)} \not\prec X^{(i)}$ for all $1 \leq i \leq n$.
\ignore{
(c)~If $1 \leq k \leq n$,
we say that $X^{(k)}$ is a \emph{broken record} at time~$n$ if it is a record but not a current record, that is, if $X^{(k)} \not\prec X^{(i)}$ for all $1 \leq i < k$ but $X^{(k)} \prec X^{(\ell)}$ for some $k < \ell \leq n$; in that case, the observation corresponding to the smallest such~$\ell$ is said to \emph{break} or \emph{kill} the
record $X^{(k)}$.
}
\end{definition}

For 
$n \geq 1$ (or $n \geq 0$, with the obvious conventions) 
let 
$\rho_n\ ( \equiv \rho_{d, n})$
denote the number of remaining records at time~$n$ (when the dimension is~$d$).
\ignore{
Note that $R_n$ and $\beta_n$ are nondecreasing in~$n$, but the same is not true for $r_n$.  For dimension $d \geq 2$, by
standard consideration of concomitants
[that is, by considering the $d$-dimensional sequence $X^{(1)}, \ldots, X^{(n)}$ sorted from largest to smallest value of (say) last coordinate]
we see that $r_n(d)$ (that is, $r_n$ for dimension~$d$, with similar notation used here for $R_n$) has, for each~$n$, the same (univariate) distribution as $R_n(d-1)$; note, however, the same equality in distribution does \emph{not} hold for the stochastic processes $r(d)$ and $R(d - 1)$.
}

\begin{definition}
\label{D:RS}
(a)~The \emph{record-setting region} at time~$n$ is the (random) closed set of points
\[
\mbox{RS}_n := \{x \in \bbR^d: 0 \leq x \not\prec X^{(i)}\mbox{\ for all $1 \leq i \leq n$}\}.
\]

(b)~We call the (topological) boundary of $\mbox{RS}_n$ (relative to the closed positive orthant determined by the origin) its \emph{frontier} and denote it by $F_n$.
\end{definition}

\begin{figure}[htb]
\begin{tikzpicture}[scale=8]
%
%
\draw[->,thick,color=black] (0,0)--(0,1.2);
\draw[->,thick,color=black] (0,0)--(1.2,0);
%
%
\draw[thick,color=black](0,.97)--
(.03,.97)--(.03,.74) --
(.23,.74)--(.23,.65) --
(.26,.65)--(.26,.62) --
(.33,.62)--(.33,.49) --
(.56,.49)--(.56,.40) --
(.62,.40)--(.62,.38) --
(.79,.38)--(.79,.25) --
(.81,.25)--(.81,.21) --
(.85,.21)--(.85,.11) --
(.96,.11)-- (.96,.0)
;
%
%
\filldraw [black]
(.03,.97) circle (.175pt)
(.23,.74) circle (.175pt)
(.26,.65) circle (.175pt)
(.33,.62) circle (.175pt)
(.56,.49) circle (.175pt)
(.62,.40) circle (.175pt)
(.79,.38) circle (.175pt)
(.81,.25) circle (.175pt)
(.85,.21) circle (.175pt)
(.96,.11) circle (.175pt)
;
\draw [black]
(.03,.74) circle (.175pt);
%
%
\draw[dashed,color=black, thick](0,1.17)--(1.17,0);
%
%
\draw[dotted,color=black, thick](0,.77)--(.77,0);
%
%
\draw[solid,color=black, thick](0,.91)--(.91,0);
%
%
\draw (.50,.55) node[color=black] {\footnotesize $F_n$};
\draw (.7,.7) node[color=black] {\footnotesize $x_+=F_n^+$};
\draw (.3,.3) node[color=black] {\footnotesize $x_+=F_n^-$};
%
%
\draw (.830,.380) node[color=black] {\footnotesize $\lambda_n$};
\draw (.035,.71) node[color=black] {\footnotesize $\tau_n$};
\draw (.17,.90) node[color=black] {\footnotesize $x_+=\widehat{F}_n^-$};
%
%
%
\draw [->,color=black](.7,.65) arc (340:330:5mm); 
\draw [->,color=black](.3,.33) arc (180:155:2mm);
\draw [->,color=black](.47,.55) arc (110:131:3mm); 
\draw [->,color=black](.17,.87) arc (325:300:2mm); 
%
%
\draw (.7,-.05) node[color=black] {\footnotesize $x_1$};
\draw (-.05,.5) node[color=black] {\footnotesize $x_2$};
\end{tikzpicture}

\caption{Record frontier $F_n$ based on $n$ observations (for some $n \geq 10$) resulting in 10 current records (shown as solid points), with the three hyperplanes $x_+ = F_n^+$, $x_+ = F_n^-$, and $x_+ = \hF_n^-$, the leading point 
$\gl_n$ and the trailing point $\tau_n$.  Concerning the three hyperplanes, see \refD{D:W} and~\eqref{hFn-def}.
}
\label{fig:frontierdefs}
\end{figure}
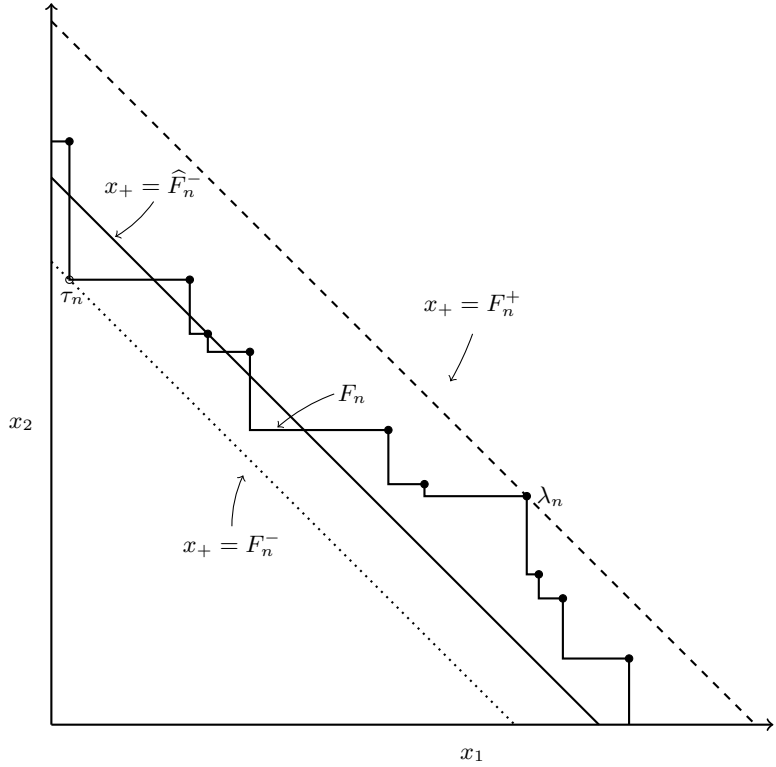

\begin{remark}
\label{R:RS}
The terminology in \refD{D:RS}(a) is natural since the next observation $X^{(n + 1)}$ sets a record if and only if it falls in the record-setting region.  Note that
\begin{align*}
\mbox{RS}_n
&= \{x \in \bbR^d: 0 \leq x \not\prec X^{(i)}\mbox{\ for all $1 \leq i \leq n$} \\
&{} \qquad \qquad \qquad \mbox{such that $X^{(i)}$ is a current record at time~$n$}\},
\end{align*}
and that the current records at time~$n$ all belong to~$\mbox{RS}_n$ but lie on its frontier.
Observe also that $F_n$ is a closed subset of $\mbox{RS}_n$.  
\end{remark}

This paper primarily concerns the stochastic process $(F_n)$, and specifically the process $F^-$ as defined (along with the process $F^+$) next
(see \refF{fig:frontierdefs}).

\begin{definition}
\label{D:W}
Recalling that $F_n$ denotes the frontier of $\mbox{RS}_n$, let
\begin{equation}
\label{-+}
F_n^- := \min\{x_+: x \in F_n\} \quad \mbox{and} \quad F_n^+ := \max\{x_+: x \in F_n\}.
\end{equation}
\ignore{
We define the \emph{width} of $F_n$ as
\begin{equation}
\label{W}
W_n := F_n^+ - F_n^-.
\end{equation}
}
Almost surely, there are for each~$n$ unique vectors $\gl_n \in F_n$ and $\tau_n \in F_n$ such that $F_n^+ = \gl_n$ and $F_n^- = \tau_n$; we refer to $\gl_n$ and $\tau_n$ as the \emph{leading} and \emph{trailing} points, respectively, of the frontier.
\end{definition}

We will need the following simple observation about the process $F^-$.

\begin{lemma}[\underline{nondecreasing sample paths for $F_n^-$}]
\label{L:-}
The process $F^-$ has nondecreasing sample paths.
\end{lemma}

\begin{proof}
The asserted monotonicity
of $F^-$ follows easily from the observation that
$F_{n + 1} \subseteq \mbox{RS}_{n + 1} \subseteq \mbox{RS}_n$.
\end{proof}

\subsection{The record-setting frontier; our two main theorems}
\label{S:boundary}

Fill and Naiman first showed, in a precise sense \cite[Theorem~1.4]{Fillboundary(2020)}, that the difference between the sum of coordinates (call it $Y_n$) of a ``generic'' current record at time~$n$ and $\L n$ converges in distribution to standard Gumbel.  They next translated results from classical extreme value theory due to Kiefer~\cite{Kiefer(1972)} to the setting of multivariate records to produce rather sharp typical-behavior and almost-sure results about the process $F^+$.  For completeness, we repeat their main result \cite[Theorem~1.8]{Fillboundary(2020)} for $F^+$ here, except that we have rather effortlessly extended part~(b) of that theorem using Kiefer's ``first proof'' as described in 
\cite[proof of Theorem~1.8(b)]{Fillboundary(2020)}.  We remark that the difference between the top-boundary threshold at about $\L n + d \L_2 n$ and bottom-boundary threshold at about 
$\L n + (d - 1) \L_2 n$ is a noteworthy feature of $F_n^+$ discussed further in \cite[Section~1.3]{Fillboundary(2020)}.  

\begin{theorem}[Kiefer~\cite{Kiefer(1972)}]
\label{T:+}
Consider the process $F^+$ defined at~\eqref{-+}.
\medskip

{\rm (a)~\underline{Typical behavior of $F^+$}:}
\[
F_n^+ - [\L n + (d - 1) \L_2 n - \L((d - 1)!)] \Lcto G.
\]
\smallskip

{\rm (b)~\underline{Top boundaries for $F^+$}:\ } For any sequence $b_n \to \infty$ which is ultimately monotone increasing,
\[
\P(F_n^+ \geq b_ n\mbox{\rm \ \io}) = \mbox{\rm $1$ or~$0$ according as $\sum e^{-b_n} b_n^{d - 1}$ diverges or converges}.
\]
In particular, for any $k \geq 2$ we have
\[
\P\left( F_n^+ \geq \L n + d \L_2 n + \sum_{i = 3}^k \L_i n + c \L_{k + 1} n\mbox{\rm \ \io} \right) =
\begin{cases}
1 & \mbox{{\rm if} $c \leq 1$;} \\
0 & \mbox{{\rm if} $c > 1$}.
\end{cases}
\]
\smallskip

{\rm (c)~\underline{Bottom boundaries for $F^+$}:}
\[
\P(F_n^+ \leq \L n + (d - 1) \L_2 n - \L_3 n - \L((d-1)!) + c\mbox{\rm \ \io}) =
\begin{cases}
1 & \mbox{{\rm if} $c \geq 0$;} \\
0 & \mbox{{\rm if} $c < 0$}.
\end{cases}
\]
\nopf
\end{theorem}
\medskip

From \refT{T:+} it follows in particular that
\ignore{
\refT{T:+} gives rise immediately to the following succinct corollary.
\medskip

\begin{corollary}[Kiefer~\cite{Kiefer(1972)}]
\label{C:+}
Consider the process $F^+$ defined at~\eqref{-+}.
\medskip

{\rm (a)~\underline{Typical behavior of $F^+$}:}
}
\[
\frac{F_n^+ - \L n}{\L_2 n} \Pto d - 1
\]
and
\ignore{
{\rm (b)~\underline{Almost sure behavior for $F^+$}:}
}
\[
\liminf \frac{F_n^+ - \L n}{\L_2 n} = d - 1 < d =
\limsup \frac{F_n^+ - \L n}{\L_2 n} \mbox{\rm \ \as}
\]
\ignore{
\end{corollary}
}
\medskip

The results derived in \cite{Fillboundary(2020)} for $F^-$ are much less sharp than for $F^+$.  For the reader's convenience, we repeat those results here.  Although parts~(a) and~(c1) were stated with coefficient $-3$ [rather than $ - (2 + c)$] for the $\L_3 n$ term, the improvement we have noted here is pointed out in \cite[Remark~3.3]{Fillboundary(2020)}.

\begin{theorem}[\cite{Fillboundary(2020)}, Theorem~1.12]
\label{T:-OLD}
Consider the process $F^-$ defined at~\eqref{-+}.
\vspace{-.1in}

{\rm (a)~\underline{Typical behavior of $F^-$}:}
\[
\P(F_n^- \leq \L n - (2 + c) \L_3 n) \to 0\mbox{\rm \ if $c > 0$}
\]
and
\[
\P(F_n^- \geq \L n + c_n) \to 0\mbox{\rm \ if $c_n \to \infty$}.
\]
\medskip

{\rm (b)~\underline{Top outer boundaries for $F^-$}:}
\[
\P(F_n^- \geq \L n + c \L_2 n\mbox{\rm \ \io}) = 0 \mbox{\rm \ if $c > 0$}.
\]
\medskip

{\rm (c1)~\underline{Bottom outer boundaries for $F^-$}:}
\[
\P(F_n^- \leq \L n - (2 + c) \L_3 n\mbox{\rm \ \io}) = 0\mbox{\rm \ if $c > 0$}.
\]
\medskip

{\rm (c2)~\underline{A bottom inner boundary for $F^-$}:}
\[
\P(F_n^- \leq \L n - \L_3 n\mbox{\rm \ \io}) = 1.
\]
\end{theorem}
\medskip

\ignore{
\refT{T:-} gives rise immediately to the following succinct corollary.
\medskip

\begin{corollary}
\label{C:-OLD}
Consider the process $F^-$ defined at~\eqref{-+}.
\medskip

{\rm (a)~\underline{Typical behavior of $F^-$}:}
\[
\frac{F_n^- - \L n}{\L_2 n} \Pto 0.
\]

{\rm (b)~\underline{Almost sure behavior for $F^-$}:\ }If $d \geq 2$, then
\[
\lim \frac{F_n^- - \L n}{\L_2 n} = 0\mbox{\rm \ \as}
\]
\end{corollary}
\medskip
}

The first of two main results of this paper, \refT{T:-}, sharpens \refT{T:-OLD} considerably.
In light of (i)~the constant-order variability for a ``generic'' current record at time~$n$ described in the opening paragraph of this subsection and (ii)~\refT{T:+}(a), we find it quite surprising that, properly centered but not scaled, $F_n^-$ has a limit in probability.
\medskip

\begin{theorem}
\label{T:-}
Consider the process $F^-$ defined at~\eqref{-+}.
\medskip

{\rm (a)~\underline{Typical behavior of $F^-$}:}
\[
F_n^- = \L n - \L_3 n - \L(d - 1) + O_p\!\left( \frac{\L_3 n}{\L_2 n} \right).
\] 
\medskip

{\rm (b)~\underline{Top outer boundaries for $F^-$}:}  If $d \geq 3$, then
\[
\P(F_n^- \geq \L n - \L_3 n - \L(d - 2) + \L 2 + c\mbox{\rm \ \io}) = 0 \mbox{\rm \ if $c > 0$}.
\]
\vspace{-.1in}

{\rm (c)~\underline{Bottom outer boundaries for $F^-$}:}
\[
\P(F_n^- \leq \L n - \L_3 n - \L d - \L 2 - c\mbox{\rm \ \io}) = 0\mbox{\rm\ if $c > 0$}.
\]
\ignore{
\medskip
{\rm (c2)~\underline{A bottom inner boundary for $F^-$}:}
\[
\P(F_n^- \leq \L n - \L_3 n\mbox{\rm \ \io}) = 1.
\]
}
\end{theorem}
\medskip

\refT{T:-} gives rise immediately to the following succinct corollary, where $O_p$ is big-oh in probability.
\medskip

\begin{corollary}
\label{C:-}
Consider the process $F^-$ defined at~\eqref{-+}.
\medskip

{\rm (a)~\underline{Typical behavior of $F^-$}:}
\[
F_n^- - (\L n - \L_3 n) \Pto - \L(d - 1)
\]
and thus
\[
\frac{F_n^- - \L n}{\L_3 n} \Pto -1
\]
and, yet more crudely,
\[
\frac{F_n^- - \L n}{\L_2 n} \Pto 0.
\]

{\rm (b)~\underline{Almost sure behavior for $F^-$}:}
\[
\lim \frac{F_n^- - \L n}{\L_2 n} = 0\mbox{\rm \ \as}
\]
Further, for fixed $d \geq 3$ we have the refinement
\[
F_n^- = \L n - \L_3 n + O(1)\mbox{\rm \ \as}
\]
\end{corollary}

\begin{remark}
We do not know how to improve \refT{T:-OLD}(b) when $d = 2$.
\end{remark}

Suppose now that instead of $F_n^-$ we consider the somewhat larger quantity
\begin{equation}
\label{hFn-def}
\hF_n^- := (\mbox{minimum coordinate-sum of any current record at time~$n$}).
\end{equation}
(See \refF{fig:frontierdefs}.)
Our second main theorem concerns the process $\hF^-$; in summary, the same results hold for $\hF^-$ as for $F^-$ in \refT{T:-}, with a sharper remainder term for $\hF^-$ in part~(a).

\begin{theorem}
\label{T:hF-}
Consider the process $\hF^-$ defined at~\eqref{hFn-def}.
\medskip

{\rm (a)~\underline{Typical behavior of $\hF^-$}:} 
\[
\hF_n^- = \L n - \L_3 n - \L(d - 1) + O_p\!\left( \frac{1}{\L_2 n} \right).
\] 
\medskip

{\rm (b)~\underline{Top outer boundaries for $\hF^-$}:\ } If $d \geq 3$, then
\[
\P(\hF_n^- \geq \L n - \L_3 n - \L(d - 2) + \L 2 + c\mbox{\rm \ \io}) = 0 
\mbox{\rm \ if\ }c > 0.
\]
\vspace{-.1in}

{\rm (c)~\underline{Bottom outer boundaries for $\hF^-$}:}
\[
\P(\hF_n^- \leq \L n - \L_3 n - \L d - \L 2 - c\mbox{\rm \ \io}) = 0\mbox{\rm\ if $c > 0$}.
\]
\end{theorem}
\medskip

As a corollary, the process $\hF^-$ satisfies the same assertions as for $F^-$ in \refC{C:-}.
\ignore{
\refT{T:hF-} gives rise immediately to the following succinct corollary.
\medskip

\begin{corollary}
\label{C:hF-}
Consider the process $\hF^-$ defined at~\eqref{hFn-def}.
\medskip

{\rm (a)~\underline{Typical behavior of $\hF^-$}:}
\[
\hF_n^- - (\L n - \L_3 n) \Pto - \L(d - 1)
\]
and thus
\[
\frac{\hF_n^- - \L n}{\L_3 n} \Pto -1
\]
and, yet more crudely,
\[
\frac{\hF_n^- - \L n}{\L_2 n} \Pto 0.
\]

{\rm (b)~\underline{Almost sure behavior for $\hF^-$}:}
\[
\lim \frac{\hF_n^- - \L n}{\L_2 n} = 0\mbox{\rm \ \as}
\]
Further, for fixed $d \geq 3$ we have the refinement
\[
\hF_n^- - = \L n - \L_3 n + O(1)\mbox{\rm \ \as}
\]
\end{corollary}
}

\begin{remark}
Combining Theorems~\ref{T:-} and \ref{T:hF-}, we find that there is little difference between the two processes in the sense that
\[
\hF_n^- - F_n^- \Pto 0,
\]
because in fact $0 \leq \hF_n^- - F_n^- = O_p\!\left( \frac{\L_3 n}{\L_2 n} \right)$.
\end{remark}

\begin{remark}
\label{R:hF-limit}
Extending \refT{T:hF-}, we conjecture that
\begin{equation}
\label{hF-limit}
(\L_2 n) \left( \hF_n^- - [\L n - \L_3 n - \L(d - 1)] \right)
\end{equation}
has a nondegenerate limiting distribution.  This is discussed further in \refR{R:rholimit}.
\end{remark}

\subsection{Outline of paper}
\label{S:outline}

The proof of \refT{T:-} relies on \refT{T:hF-}, so we tackle the latter first.  In Sections
 \ref{S:LBh-}--\ref{S:UBh-} we apply the first moment method and the second moment method, respectively, to the number of remaining records with suitably small coordinate-sum; this leads to the proof of \refT{T:hF-} in \refS{S:proofh-}.  In Sections~\ref{S:char}--\ref{S:expected} we review and extend the theory of generators developed in~\cite{Fillgenerating(2018)}.  In \refS{S:LB-} we apply the first moment method to the number of generators with suitably small coordinate sum; this, together with the upper bounds on $\hF^-$ in \refT{T:hF-}, leads to the proof of \refT{T:-} in \refS{S:proof-}.
 
\begin{remark}
\label{R:compare}
Because $F_n^- \leq \hF_n^-$, \refT{T:-}(b) follows immediately from \refT{T:hF-}(b), as does \refT{T:hF-}(c) from \refT{T:-}(c).
\end{remark}

{\bf More notation:\ }Throughout the paper, the boundaries we consider will without exception have the form
\begin{equation}
\label{bndefhat}
b_n := \L n - \L_3 n - \L c_n\mbox{\ with $c_n > 0$ and $c_n = \Theta(1)$}.
\end{equation}
Also, we will often use the notation
\begin{equation}
\label{gbdef}
\gb_n := n e^{ - b_n}
\end{equation}
The dimension $d \geq 2$ will always remain fixed as $n \to \infty$.

\section{Stochastic lower bound on $\hF^-_n$ via the first moment method}
\label{S:LBh-}

In this section we show how to obtain a suitable stochastic lower bound on $\hF_n^-$.  See \refP{P:h-lower} for the result.  The idea, for a suitably chosen sequence $(b_n)$ is to apply the first moment method (computation of sufficiently small mean, together with application of Markov's inequality) to the count 
$\rho_n(b_n)$, where
\begin{equation}
\label{rhonbndef}
\rho_n(b) := \#\{ \mbox{remaining records~$r$ at epoch~$n$ with $r_+ \leq b$} \}.
\end{equation}
Asymptotic determination of the mean is obtained by 
suitably modifying the asymptotic determination of the mean of $\rho_n = \rho_n(\infty)$ in 
\cite[Section~2]{Bai(2005)}.

\subsection{Upper (and lower) asymptotic bound(s) on mean}
\label{S:mean hat}

In the next lemma we determine detailed asymptotics for the mean of 
$\rho_n(b_n)$ when $(b_n)$ is a boundary of interest in establishing Theorems~\ref{T:-} and~\ref{T:hF-}.  The proof is rather elementary, but we defer it to \refApp{A:meanproof}.  We define
\begin{equation}
\label{Jjdef}
J_j(x) := \int_x^{\infty} (\L z)^j e^{-z} \dd z.
\end{equation}
and note that $J_j(x) \sim (\L x)^j e^{-x}$ as $x \to \infty$.

\begin{lemma}
\label{L:mean hat}
With the notation and assumptions of~\eqref{bndefhat}--\eqref{gbdef} and~\eqref{Jjdef}, as $n \to \infty$ we have
\begin{align}
\lefteqn{\E \rho_n(b_n)} \nonumber \\ 
&= [1 + O(n^{-1} (\L_2 n)^2)]
\frac{1}{(d - 1)!} \sum_{j = 0}^{d - 1} (-1)^j \binom{d - 1}{j} (\L n)^{d - 1 - j} J_j(\gb_n), \label{meanhat}
\end{align}
or, equivalently,
\begin{align}
\E \rho_n(b_n)
&= \frac{1}{(d - 1)!} \sum_{j = 0}^{d - 1} (-1)^j \binom{d - 1}{j} (\L n)^{d - 1 - j} J_j(\gb_n) \nonumber \\ 
&{} {} \qquad \qquad + O(n^{-1} (\L n)^{d - 1 - c_n} (\L_2 n)^2). \label{meanhatalt} 
\end{align}
\end{lemma}

\begin{remark}
We need only lead-order asymptotics for the mean in this section, but (as we shall see in the proof of~\refL{L:UB variance}) we require much more detailed asymptotics for it in the next section---asymptotics with an additive $o(1)$ remainder term, as we have in~\eqref{meanhatalt}.
\end{remark}

We are now in position to apply Markov's inequality to bound the probability of the event $\{\hF^-_n \leq b_n\} = \{\rho_n(b_n) \geq 1\}$.

\begin{proposition}[\underline{Stochastic lower bound on $\hF_n^-$}]
\label{P:h-lower}
With the notation and assumptions of~\eqref{bndefhat}, as $n \to \infty$ we have
\begin{align*}
\P(\hF^-_n \leq b_n) 
&\leq \E \rho_n(b_n) 
= (1 + o(1))\,\frac{1}{(d - 1)!} (\L n)^{d - 1 - c_n}.
\end{align*}
\end{proposition}

\section{Stochastic upper bound on $\hF^-_n$ via 
second moment method}
\label{S:UBh-}

In this section we show how to obtain a suitable stochastic upper bound on $\hF^-_n$ 
(and thus also on $F^-_n$).  
See \refP{P:-upper} for the result.  The idea, for 
a suitably chosen sequence $(b_n)$, is to apply the second moment method (computation of sufficiently large mean and sufficiently small variance, together with application of Chebyshev's inequality) to the count 
$\rho_n(b_n)$ [recall the definition~\eqref{rhonbndef}], which almost surely equals
\begin{equation}
\label{rhoonbndef}
\rho^{\circ}_n(b_n) := \#\{ \mbox{remaining records~$r$ at epoch~$n$ with $r_+ < b_n$} \}.
\end{equation}
For the mean, we will use \refL{L:mean hat}.
The bound on the variance of $\rho_n(b_n)$ is obtained by suitably modifying the already quite technical asymptotic determination of the variance of $\rho_n = \rho_n(\infty)$ in \cite[Section~2]{Bai(2005)}; the determination here is quite a bit more technical still.

\subsection{Upper bound on variance}
We next show that the standard deviation of $\rho_n(b_n)$ is of smaller order of magnitude than the 
mean---and by enough so that our proof (in \refS{S:proofh-}) of \refT{T:-}(b) (for $\hF^-$, which implies the result for $F^-$) using the first Borel--Cantelli lemma will succeed.  The rather long and rather computationally technical proof of the following result is deferred to \refApp{A:varproof}, where the reverse inequality (not needed in this paper) is also established.

\begin{lemma}
\label{L:UB variance}
With the notation and assumptions of~\eqref{bndefhat}, as $n \to \infty$ we have
\begin{equation}
\label{varUB}
\Var \rho_n(b_n) \leq (1 + o(1)) \E \rho_n(b_n).
\end{equation}
\end{lemma}

\subsection{Stochastic upper bound on $\hF_n^-$}
\label{S: UB hFn-}

We are now in position to utilize Chebyshev's inequality to provide a bound
on $\P(\hF^-_n \geq b_n) = \P(\rho^{\circ}_n(b_n) = 0) = \P(\rho_n(b_n) = 0)$.

\begin{proposition}[\underline{Stochastic upper bound on $\hF_n^-$}]
\label{P:-upper}
With the notation and assumptions of~\eqref{bndefhat}, as $n \to \infty$ we have
\begin{align*}
\P(F^-_n \geq b_n) 
\leq \P(\hF^-_n \geq b_n) 
&\leq (1 + o(1))\,(d - 1)! (\L n)^{ - (d - 1 - c_n)} \\
&= O((\L n)^{ - (d - 1 - c_n)}).
\end{align*}
\end{proposition}

\begin{proof}
The first asserted inequality follows because $F^-_n \leq \hF^-_n$.  Moreover,
using Chebyshev's inequality, \refL{L:UB variance}, and \refL{L:mean hat}, we find
\begin{align*}
\P(\hF^-_n \geq b_n) 
&\leq \P(\rho_n(b_n) = 0) = \P(\rho_n(b_n) - \E \rho_n(b_n) \leq - \E \rho_n(b_n)) \\
&\leq \frac{\Var \rho_n(b_n)}{[\E \rho_n(b_n)]^2}
\leq (1 + o(1))\,[\E \rho_n(b_n)]^{-1} \\
&= (1 + o(1))\,(d - 1)! (\L n)^{ - (d - 1 - c_n)} \\
&= O((\L n)^{ - (d - 1 - c_n)}),
\end{align*}
as desired.
\end{proof}

\begin{remark}
\label{R:rholimit}
\refL{L:UB variance} and
the reverse inequality of \refR{R:varlower} suggest that the law of $\rho_n(b_n)$ might be well approximated by a Poisson distribution with the same mean, but, after attempts using the Stein--Chen method (see, \eg,\ \cite{Barbour_H_J_1992}) or the method of moments, we have been unable to prove such an approximation even in the case that $\E \rho_n(b_n)$ has a limit $\gl \in (0, \infty)$.  For fixed~$a \in \bbR$, let $R_n(a)$ denote $\rho_n(b_n)$ when
\begin{equation}
\label{bna}
b_n = \L n - \L_3 n - \L(d - 1) + \frac{a}{\L_2 n},
\end{equation}
\ie,\ when $c_n = (d - 1) e^{ - a / \L_2 n}$ in~\eqref{bndefhat}.  Even if a Poisson approximation should fail, we certainly conjecture that $R_n(a)$ converges in distribution to a nondegenerate $R(a)$ as $n \to \infty$ with 
$\P(R(a) = 0)$ continuous and strictly decreasing in~$a$.  In that case, it follows that~\eqref{hF-limit} has limiting distribution function $a \mapsto \P(R(a) \geq 1)$.

In particular, if $R(a)$ is Poisson distributed for every~$a$, then~\eqref{hF-limit} converges in distribution to 
$ - G^*$, where $G^*$ has a Gumbel distribution with location $ - \frac{\L[(d - 1)!]}{d - 1}$ and scale 
$\frac{1}{d - 1}$. 
\end{remark}

\section{Proof of \refT{T:hF-}}
\label{S:proofh-}
In this section we prove \refT{T:hF-}.

\begin{proof}[Proof of \refT{T:hF-}]
\ \vspace{-.1in}\\

(a)~This follows readily from Propositions~\ref{P:h-lower} and \ref{P:-upper}.  Here are some details.  For $a \in \bbR$, let
\begin{equation}
\label{bna again}
b_n(a) := \L n - \L_3 n - \L(d - 1) + \frac{a}{\L_2 n},
\end{equation}
as at~\eqref{bna}; this is an instance of~\eqref{bndefhat} with $c_n = (d - 1) e^{ - a / \L_2 n}$.  By \refP{P:h-lower},
\[
\P(\hF^-_n \leq b_n(a)) \leq (1 + o(1)) \frac{1}{(d - 1)!} (\L n)^{d - 1 - c_n} \to \frac{1}{(d - 1)!} e^{(d - 1) a};
\]
the last expression here tends to $0$ as $a \to - \infty$.  Similarly, by \refP{P:-upper},
\[
\P(\hF^-_n \geq b_n(a)) \leq (1 + o(1)) (d - 1)! (\L n)^{- (d - 1 - c_n)} \to (d - 1)! e^{- (d - 1) a},
\]
and the last expression here tends to $0$ as $a \to \infty$.  It follows that the sequence of distributions of~\eqref{hF-limit} is tight, \ie,\ that \refT{T:hF-}(a) holds. 

(b)~Like $F^-$ (\refL{L:-}), the process $\hF^-$ has nondecreasing sample paths.  From this it follows that if 
$(b_n)$ is (ultimately) monotone nondecreasing and $(n_j)$ is any strictly increasing sequence of positive integers, then
\[
\{\hF_n^- \geq b_n\mbox{\rm\ \io$(n)$}\} \subseteq \{\hF_{n_{j + 1}}^- \geq b_{n_j}\mbox{\rm\ \io$(j)$}\}.
\]
To complete the proof of part~(b), we choose $b_n \equiv \L n - \L_3 n - \L(d - 2) + \L 2 + c$ with 
$c > 0$ and $n_j \equiv 2^j$, bound $\P(\hF_{n_{j + 1}}^- \geq b_{n_j})$ using 
\refP{P:-upper}, and apply the first Borel--Cantelli lemma.

Here are the details.  If $n$ is even, then 
\begin{align*}
b_{n / 2} 
&= \L(n / 2) - \L_3(n / 2) - \L(d - 2) + \L 2 + c \\
&= \L n - \L_3(n / 2) - \L(d - 2) + c\\ 
&\geq \L n - \L_3 n - \L(d - 2) + c,
\end{align*} 
the last expression being the one in \refP{P:-upper} with $c_n \equiv e^{-c} (d - 2)$.  Thus, by that proposition,
\begin{align*}
\P(\hF_{n_{j + 1}}^- \geq b_{n_j}) 
&= \P(\hF_{n_{j + 1}}^- \geq b_{n_{j + 1} / 2}) \\
&\leq \P(\hF_{n_{j + 1}}^- \geq \L n_{j + 1} - \L_3 n_{j + 1} - \L(d - 2) + c) \\
&= O((\L n_{j + 1})^{- [d - 1 - e^{-c} (d - 2)]}) = O((j + 1)^{- [1 + (1 - e^{-c}) (d - 2)]}),  
\end{align*}
which is summable.

(c)~To prove part~(c) [which, as noted in \refR{R:compare}, will also follow immediately once we prove \refT{T:-}(c)], we begin with an argument similar to that for part~(b).  
If $(b_n)$ is (ultimately) monotone nondecreasing and $(n_j)$ is any strictly increasing sequence of positive integers, then
\[
\{F_n^- \leq b_n\mbox{\rm\ \io$(n)$}\} \subseteq \{F_{n_j}^- \leq b_{n_{j + 1}}\mbox{\rm\ \io$(j)$}\}.
\]
To complete the proof of part~(c), we choose $b_n \equiv \L n - \L_3 n - \L d - \L 2 - c$ with 
$c > 0$ and $n_j \equiv 2^j$, bound $\P(F_{n_j}^- \leq b_{n_{j + 1}})$ using \refP{P:h-lower}, and apply the first Borel--Cantelli lemma.

Here are the details.  First note that
\[
b_{2 n} = \L(2 n) - \L_3(2 n) - \L d - \L 2 - c \leq \L n - \L_3 n - \L d - c,
\] 
the bounding expression being the one in \refP{P:h-lower} with $c_n \equiv e^c d$.  Thus, by that proposition,
\begin{align*}
\P(F_{n_j}^- \leq b_{n_{j + 1}})
&\leq \P(F_{n_j}^- \leq \L n_j - \L_3 n_j - \L d  - c) \\
&= O((\L n_j)^{- [e^c d - (d - 1)]}) = O(j^{- [1 + (e^c - 1) d]}), 
\end{align*}
which is summable.
\end{proof}

\section{Characterization of generators}
\label{S:char}

The 
unpublished manuscript~\cite{Fillgenerating(2018)} by Fill and Naiman developed the concept of generators of multivariate records mainly in connection with an importance-sampling algorithm for generating (simulating) records.  We shall find the same concept crucial for our improvement \refT{T:-}(c) to 
\refT{T:-OLD}(c2), the latter of which was established using a quite different idea, namely, a certain geometric lemma \cite[Lemma~3.1]{Fillboundary(2020)}.  Accordingly, in this section and the next we review and extend the theory of generators developed in~\cite{Fillgenerating(2018)}.  In this section we provide a characterization of the set of generators that is useful in counting them.

\begin{definition}
Suppose $x \in [0, \infty)^d$.
\smallskip

\noindent
(a)~The \emph{closed positive orthant generated} (or \emph{determined}) \emph{by~$x$} is the set
\[
O^+_x := \{y \in [0, \infty)^d: y \geq x\}.
\]
\ignore{
\noindent
(b)~The 
\marginal{Delete if not used.}
\emph{open negative orthant generated} (or \emph{determined}) \emph{by~$x$} is the set
\[
O^-_x := \{y \in [0, \infty)^d: y \prec x\}.
\]
}

\noindent
(b)~The minimum points of the frontier $F_n$ are called \emph{generators}.  We denote the set of generators at time~$n$ by $G_n$.
\end{definition}

\begin{remark}
(a) The record-setting region $\mbox{RS}_n$ equals the union $\cup_{g \in G_n} O_g^+$ of closed positive orthants.  The elements of $G_n$ are called generators because $\mbox{RS}_n$ is the up-set in 
$[0, \infty)^d$ generated by $G_n$ with respect to the partial order~$\leq$.

(b) The almost surely unique generator with minimum coordinate-sum is the trailing point $\tau_n$, just as the remaining record with maximum coordinate-sum is the leading point $\gl_n$.
\end{remark}

There are $11$ generators in \refF{fig:frontierdefs}, including the trailing point $\tau_n$ at the intersection of $F_n$ and the dotted hyperplane (line) marked with $x_+ = F_n^-$.  In terminology we shall establish shortly, $9$ of these are interior (\ie,\ $2$-dimensional) generators and $2$ of them are $1$-dimensional generators.
\medskip 

We now proceed to characterize the set of generators.

Denote the $\rho \equiv \rho_n$ current records at a given time~$n$ by
$r^{(1)}, \dots, r^{(\rho)}$ (listed here in arbitrary, but fixed, order).
\ignore{
\smallskip

{\bf Assumption:\ }When discussing the deterministic geometry of such points $r^{(1)}, \dots, r^{(\rho)}$, 
we will for simplicity assume (as is almost surely true throughout time in our records model) that there are no ties in any coordinate, that is, that
\begin{equation}
\label{no coord ties}
\mbox{For each $j \in [d]$, the $\rho$ values $r^{(i)}_j$, $i \in [\rho]$, are distinct.}
\end{equation}
We will also assume that $r^{(i)} \in (0, 1)^d$ for $i \in [\rho]$, and that these points are (pairwise) incomparable; because we assume~\eqref{no coord ties}, for ``incomparable'' here we don't need to specify whether the partial order is $\preceq$ or $\leq$.
\smallskip
}
The record-setting region $S \equiv \mbox{RS}_n$ is then the closed set
\ignore{
\begin{equation}
\label{RS}
S = \cap_{i = 1}^{\rho} [O^-(r^{(i)})]^c.
\end{equation}

Our development here begins by recalling~\eqref{RS}:
}
\begin{align}
S 
&= \cap_{i = 1}^{\rho} \left[ \cup_{k = 1}^d O^+\!\left( r^{(i)}_k e^{(k)} \right) \right]
= \cup_{k_1 = 1}^d \cdots \cup_{k_{\rho} = 1}^d \cap_{i = 1}^{\rho} O^+\!\left( r^{(i)}_{k_i} e^{(k_i)} \right) \nonumber \\
&= \cup_{k_1 = 1}^d \cdots \cup_{k_{\rho} = 1}^d O^+\!\left( \vee_{i = 1}^{\rho} r^{(i)}_{k_i} e^{(k_i)} \right) \nonumber \\
\label{r}
&= \cup_{k \in [d]^{[\rho]}}\,O^+\!\left( R^{(\Pi_1(k))}_1, \ldots, R^{(\Pi_d(k))}_d \right),
\end{align}
where for $j \in [d]$ and $k \in [d]^{[\rho]}$ we have defined the ordered partition
$\Pi(k) = (\Pi_1(k), \dots, \Pi_d(k))$ of $[\rho]$ by
\[
\Pi_j(k) := k^{-1}(\{j\}) = \{i \in [\rho]:k_i = j\},
\]
and for $j \in [d]$ and $P \subseteq [\rho]$ we have defined
\[
R_j^{(P)} := \vee_{i \in P}\,r^{(i)}_j.
\]

Therefore we have the neat representation
\begin{equation}
\label{rep}
S = \bigcup O^+\!\left( R_1^{(\Pi_1)}, \ldots, R_d^{(\Pi_d)} \right),
\end{equation}
where the union here is taken over all ordered partitions $\Pi = (\Pi_1, \dots, \Pi_d)$ of $[\rho]$ into~$d$ sets; each $\Pi_j$  is allowed to be empty, in which case $R^{(\Pi_j)}_j := 0$.
This shows immediately that every element of $G \equiv G_n$ has in each coordinate either~$0$ or the value of some record in that coordinate.

To simplify our characterization of generators, we begin by considering only ``interior'' generators.
For any point $x \in O^+_0$, let 
$\nu(x)$ denote the set of non-zero coordinates of~$x$, and observe that~$x$ lies in the interior of $O^+_0$ if and only if $\nu(x) = [d]$.  We call such a point~$x$ an \emph{interior} point.

Observe that a point~$x$ of the form $\left( R_1^{(\Pi_1)}, \ldots, R_d^{(\Pi_d)} \right)$ appearing in~\eqref{rep} is interior if and only if all the cells $\Pi$ of the partition are nonempty.  Next, note that $x \in (0, \infty)^d$ is of such a form if and only if there exist~$d$ distinct indices $i_1, \dots, i_d$ such that $x_j = r^{(i_j)}_j$ for $j \in [d]$.

We are now in position to state and prove a characterization of the set~$I$ of interior generators.  (Note that $I \subset G \subset S$.)

\begin{theorem}
\label{T:interior}
A point $g \in [0, \infty)^d$ belongs to~$I$ if and only if
\begin{enumerate}
\item[(i)] $g \in S$, and
\item[(ii)] there exist~$d$ distinct indices $i_1, \dots, i_d$ such that
\begin{equation}
\label{min}
g_j = r^{(i_j)}_j = \min\left\{ r^{(i_{\ell})}_j: \ell \in [d] \right\}\mbox{\ \rm for every $j \in [d]$}.
\end{equation}
\end{enumerate}
\end{theorem}

\begin{proof}
First suppose $g \in I$.  Then~(i) is automatic from the definition of~$I$.  Moreover, we know from our earlier discussion that~(ii) holds for $g = \left( R_1^{(\Pi_1)}, \ldots, R_d^{(\Pi_d)} \right)$ with the possible exception of the second equality in~\eqref{min}.  But if that equality does not hold, let $j, \ell \in [d]$ with $j \neq \ell$ satisfy
\[
r^{(i_{\ell})}_j < R_j^{(\Pi_j)}.
\]
We then move $i_{\ell}$ from the cell $\Pi_{\ell}$ to the cell $\Pi_j$ in order to form a new partition, call it $\Pi'$.  Then
\[
g > \left( R_1^{(\Pi'_1)}, \ldots, R_d^{(\Pi'_d)} \right) \in S,
\]
so~$g$ is not a generator.

Next we prove the converse.
If~$g$ has these two properties, then $g \in (0, \infty)^d$ belongs to~$S$, so all that is left to show is that~$g$ is a minimum (with respect to~$\leq$) of~$S$.  Suppose that $x < g$;
we will complete the proof by showing that $x \not\in S$.

Let $j_0$ satisfy $x_{j_0} < g_{j_0}$.  Then
\begin{equation}
\label{j0}
x_{j_0} < g_{j_0} = r^{(i_{j_0})}_{j_0}
\end{equation}
using~\eqref{min} for the equality.  Additionally, for $j \neq j_0$ we have
\begin{equation}
\label{j}
x_j \leq g_j < r_j^{(i_{j_0})},
\end{equation}
where the second inequality holds by~\eqref{min} because
\[
g_j = r_j^{(i_j)} = \min\left\{ r^{(i_{\ell})}_j: \ell \in [d] \right\},
\]
which almost surely is strictly smaller than $r_j^{(i_{j_0})}$ 
because $i_j \neq i_{j_0}$.  Combining~\eqref{j0} and~\eqref{j}, we see that $x \prec r^{(i_{j_0})}$, and so $x \not\in S$.
\end{proof}

\begin{remark}
\refT{T:interior} gives an injection from the set of interior generators into the set of ordered $d$-tuples of remaining records.
\end{remark}

Now that we have characterized the interior generators, it is straightforward to characterize~$G$ in terms of projections of the current records to lower-dimensional coordinate subspaces, but some care must be taken to ensure that
the almost sure property of having no coordinate ties
remains true after projection.  To begin a careful description, given a subset $T = \{j_1, \dots, j_t\}$ of 
$[d]$ with $|T| = t \in [d]$ and $1 \leq j_1 < \cdots < j_t \leq d$, define the \emph{projection mapping} $\pi_T: \bbR^d \to \bbR^{t}$ by
\[
\pi_T(x_1, \dots, x_d) := (x_{j_1}, \dots, x_{j_t}),
\]
and define the \emph{injection mapping} $\iota_T: \bbR^t \to \bbR^d$ by
\[
\iota_T(x_1, \dots, x_t) := \vee_{k = 1}^t x_{j_k} e^{(j_k)},
\]
where $\vee$ denotes coordinate-wise maximum.
Recall that $\nu(x)$ denotes the set of nonzero coordinates of a point $x \in [0, \infty)^d$.  Define the set of \emph{$T$-generators} to be the set
\[
G_T := G \cap \{x:\,\nu(x) = T\}
\]
and observe that~$G$ is the disjoint union
\[
G = \cup_{T \subseteq [d]} G_T.
\]
This observation, together with a characterization of each $G_T$, thus provides a characterization of~$G$.  A characterization of each $G_T$ is obtained by combining the following theorem with \refT{T:interior}.

To set up the statement of the theorem, consider the image
\[
R_T := \pi_T(R) = \left\{ \pi_T(r^{(i)}): i \in [\rho] \right\} \subset \bbR^{|T|}
\]
under $\pi_T$ of the set $R := \left\{ r^{(i)}: i \in [\rho] \right\}$ of current records, and note that $R_T$ inherits the property
of ``no ties in any coordinate'' from~$R$.  Let $I_T$ denote the set of interior generators of $R_T$, and let $G'_T := \iota_T(I_T)$ denote the injection of $I_T$ into $\bbR^d$.

\begin{theorem}
\label{T:T}
For every $T \subseteq [d]$ we have $G_T = G'_T$.
\end{theorem}

\begin{proof}
As above, let $t = |T|$.
There is no loss of generality (and there is some ease in notation) in supposing that $T = [t]$, and thus $x \in G_T$ if and only if $x \in G$ and $x_{t + 1} = \cdots = x_d = 0$.
Let $x = (x_1, \dots, x_t, 0, \dots, 0)$ satisfy $\nu(x) = t$.  We will show that $x \in G_T$---equivalently, that
$x \in G$---if and only if $\pi_T(x) \in I_T$---equivalently, that $x \in \iota_T(I_T) = G'_T$.

Indeed, for~$x$ to be a generator, there are two requirements: (i)~$x \in S$, and (ii)~$x$ is a minimum of~$S$.  The requirement~(i) is that for each~$i$ there should exist $j \in [d]$ such that $x_j \geq r^{(i)}_j$.  However, since we assume that $r^{(i)} \succ 0$, such~$j$ must belong to $[t]$.  We have thus argued that
$x$ is in $S = \mbox{RS}(R)$ (the record-setting region determined by the points in~$R$) if and only if 
$\pi_T(x) \in \mbox{RS}(R_T)$.

The requirement~(ii) is that $y < x$ must imply $y \notin S$.  But note that $y < x$ if and only if $y$ is of the form $y = (y_1, \dots, y_t, 0, \dots, 0)$ with $\pi_T(y) < \pi_T(x)$.  Thus requirement~(ii) can be rephrased thus:\ If $y = (y_1, \dots, y_t, 0, \dots, 0)$ with $\pi_T(y) < \pi_T(x)$, then $y \notin \mbox{RS}$---equivalently, by what we argued in connection with requirement~(i), that $\pi_T(y) \notin \mbox{RS}(R_T)$.

So we have argued that~$x$ is a generator if and only if $\pi_T(x) \in I_T$, \ie,\ if and only if $x \in G'_T$.  This is as desired.
\end{proof}

In light of \refT{T:T}, we call the number of nonzero coordinates of a generator its \emph{dimension}.
\refF{F:frontier_3d} shows the generators
of various dimensions for an example with $d=3.$

\begin{figure}[htb]
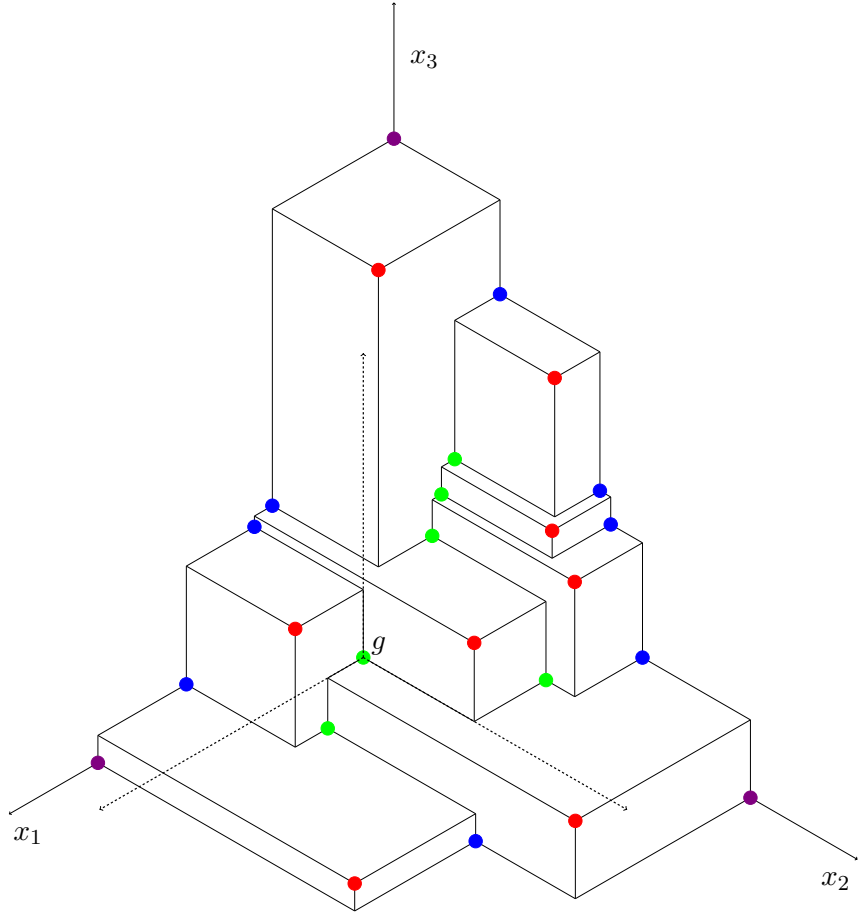

\include{figure1}
\caption{Example of a record frontier in dimension $d=3$ with $\rho=8$ remaining records shown in red and the resulting 
$\gamma = 17$ generators:\
three one-dimensional generators shown in violet, eight two-dimensional generators shown in blue, and six three-dimensional (interior) generators shown in green.
The lower boundary of one of the orthants $O^+_g$ is shown using dashed lines.}
\label{F:frontier_3d}
\end{figure}

\begin{example}
Suppose
$d = 4$ and the current records are $(2, 8, 3, 7)$ and $(5, 1, 4, 6)$.  Then $|G| = 8$, because $|G_T| = 1$ for precisely eight nonempty subsets~$T$ of $[4]$ and $|G_T| = 0$ otherwise.  The eight subsets~$T$ for which $|G_T| = 1$ are
\begin{align*}
G_{\{1\}} &= \{(5, 0, 0, 0)\};\,G_{\{2\}} = \{(0, 8, 0, 0)\};\,G_{\{3\}} = \{(0, 0, 4, 0)\}; \\
G_{\{4\}} &= \{(0, 0, 0, 7)\};\,
G_{\{1, 2\}} = \{(2, 1, 0, 0)\};\,G_{\{1, 4\}} = \{(2, 0, 0, 6)\}; \\
G_{\{2, 3\}} &= \{(0, 1, 3, 0)\};\,
G_{\{3, 4\}} = \{(0, 0, 3, 6)\}.
\end{align*}
Thus there are four one-dimensional generators, four two-dimensional generators, and no generators with dimension exceeding two.
\end{example}

\section{The expected number of generators}
\label{S:expected}

The proof of \refT{T:-}(c) requires a tight upper bound on the expected number of generators at time~$n$ with suitably small coordinate-sum.  In this section we warm up with a result of independent interest, giving an asymptotic approximation for 
the expected total number of generators at time~$n$.  We remark in passing that such an approximation proves useful in the analysis of the importance-sampling record-generating scheme described in \cite[Sections 2--4]{Fillgenerating(2018)}.  In this section, we abbreviate
$\dd x_1 \cdots \dd x_d$ as $\dd \xx$. 

\subsection{Exact expressions}
\label{S:exact}

Let $\gamma_{d, n}$ (respectively, $\iota_{d, n}$) denote the number of generators (resp.,\ interior generators) after a given number~$n$ of $d$-dimensional observations.  Our first result relates the expectations of these two quantities.

\begin{lemma}
\label{L:GI}
For integers $d \geq 0$ and $n \geq 0$, we have
\begin{equation}
\label{Gdn}
\E \gamma_{d, n} = \sum_{k = 0}^d {d \choose k} \E \iota_{k, n},
\end{equation}
\end{lemma}

\begin{proof}
This is immediate from \refT{T:T} and preceding discussion.
\end{proof}

In \refL{L:GI}, note that $\iota_{0, n} = \delta_{0, n}$: There is a single $0$-dimensional generator (namely, the origin in $\bbR^d$) if $n = 0$ and no $0$-dimensional generators otherwise.
Also note that $\iota_{d, n} = 0$ if $n < d$.

The next result gives an exact expression for $\E \iota_{d, n}$ for $n \geq d \geq 1$.
We write $n^{\underline{k}}$ for the falling factorial power
\[
n (n - 1) \cdots (n - k + 1) = k! \mbox{${n \choose k}$}.
\]

\begin{lemma}
\label{L:exact}
For integers $n \geq d \geq 1$, we have
\begin{equation}
\label{Idn}
\E \iota_{d, n}
= n\fall{d} \int_{(0, 1]^d}\,x_{\times}^{d - 1} (1 - x_{\times})^{n - d} \dd \xx.
\end{equation}
\end{lemma}

\begin{proof}
To facilitate the statement and proof of \refL{L:exact}, and in order to follow more closely the analogous treatment of remaining records in \cite[Section~2]{Bai(2005)}, we may and do switch from Exponential$(1)$ observation coordinates to observations uniformly distributed in $[0, 1)^d$.
Referring to \refT{T:interior}(ii), let us say that the $d$-tuple $(X^{(i_1)}, \ldots, $ $X^{(i_d)})$ of observations (where the indices $i_j$ are distinct elements of $[n]$) \emph{generates} an epoch-$n$ interior generator~$g$ if those~$d$ observations are all remaining records at epoch~$n$ and 
\[
g_j = X^{(i_j)}_j = \min\{X^{(i_{\ell})}_j: \ell \in [d]\}\mbox{\ for every $j \in d$}.
\]
Note that every interior generator is generated by precisely one such generating $d$-tuple.  Thus 
$\E \iota_{d, n}$ equals $n\fall{d}$ times the probability that $(X^{(1)}, \ldots, X^{(d)})$ generates an interior generator.  Condition on the value $\yy := (x^{(1)}, \ldots, x^{(d)})$ of this $d^2$-tuple.  According to \refT{T:interior}, in order for $\yy$ to generate an interior generator, two conditions are required.  One is that
\begin{equation}
\label{ellj}
x^{(\ell)}_j \geq x^{(j)}_j \mbox{\ for every $\ell, j \in d$ with $\ell \neq j$}.
\end{equation}
Let $x := \left( x^{(1)}_1, \ldots, x^{(d)}_d \right)$.
The other condition is that the remaining $n - d$ observations each need to fall outside $O^+_x$, guaranteeing the condition $x \in S$ required by \refT{T:interior}(i).

Therefore,
\[
\E \iota_{d, n}
= n\fall{d} \int_{\yy:\mbox{\scriptsize \eqref{ellj} holds}} \left[ 1 - \sprod (1 - x^{(j)}_j) \right]^{n - d} \dd \yy,
\]
a $d^2$-dimensional integral which reduces effortlessly to a $d$-dimensional integral:
\begin{align*}
\E \iota_{d, n}
&= n\fall{d} \int_{[0, 1)^d}\,\left[ \sprod (1 - x_j) \right]^{d - 1} \left[ 1 - \sprod (1 - x_j) \right]^{n - d} \dd \xx \\
&= n\fall{d} \int_{(0, 1]^d}\,x_{\times}^{d - 1} (1 - x_{\times})^{n - d} \dd \xx,
\end{align*}
as desired.
\end{proof}

\begin{remark}
(a)~The exact expression~\eqref{Idn} in \refL{L:exact} may be compared to a similar expression for 
$\E \rho_{d, n}$ derived in \cite[Section~2]{Bai(2005)}:\ For $d \geq 1$ and $n \geq 1$ we have
\begin{equation}
\label{rhodn}
\E \rho_{d, n} = n \int_{(0, 1]^d}\,(1 - x_{\times})^{n - 1} \dd \xx.
\end{equation}
In fact, by expanding the factor $x_{\times}^{d - 1}$ appearing in the integrand in~\eqref{Idn} as 
\[
[1 - (1 - x_{\times})]^{d - 1} = \sum_{j = 0}^{d - 1} (-1)^j \binom{d - 1}{j} (1 - x_{\times})^j,
\] 
one sees that the expected counts of interior generators and expected counts of remaining records are related by
\begin{align*}
\E \iota_{d, n}
&= n\fall{d} \sum_{j = 0}^{d - 1} (-1)^j \frac{\binom{d - 1}{j}}{n - d + j + 1} 
\E \rho_{d, n - d + j + 1}
\end{align*} \\
for $n \geq d \geq 1$.  But we do not know of any use for this connection.

(b)~An alternative expression to~\eqref{rhodn} is
\[
\E \rho_{d, n} = \sum_{j = 1}^n (-1)^{j - 1} \binom{n}{j} j^{ - (d - 1)} =: \widehat{H}^{(d - 1)}_n,
\]
a so-called \emph{Roman harmonic number} studied by \cite{Loeb(1989)}, \cite{Roman(1992)}, \cite{Sesma(2017)}.
\end{remark}

\subsection{Asymptotics}
\label{S:asymptotics}

From here we follow the same outline as for the expected number of remaining records in Bai et al.~\cite{Bai(2005)} to obtain an asymptotic expansion for $\E \iota_{d, n}$ (see our \refT{T:Gdnasymptotics}, the main result of \refS{S:expected}).  Accordingly, we begin by considering a Poissonized analogue of $\E \iota_{d, n}$.

\begin{lemma}
\label{L:Poissonized}
For integers $d \geq 1$ and $n \geq 0$, define
\[
\hat{\iota}_{d, n}
:= n^d \int_{[0, 1)^d} x_{\times}^{d - 1} \exp(- n x_{\times}) \dd \xx.
\]
Then, for fixed~$d$, as $n \to \infty$ we have
\[
\hat{\iota}_{d, n}
= (\L n)^{d - 1} \sum_{j = 0}^{d - 1} \frac{(-1)^j \Gamma^{(j)}(d)}{j! (d - 1 - j)!} (\L n)^{-j}
+ O((n \L n)^{d - 1} e^{-n}).
\]
\end{lemma}

\begin{proof}
We have the following simple derivation:
\begin{align*}
\hat{\iota}_{d, n}
&= \int_{[0, n^{1/d})^d} u_{\times}^{d - 1} \exp(- u_{\times}) \dd \uu
\mbox{\quad ($x_j \equiv n^{-1/d} u_j$)} \\
&= \int_{(-d^{-1} \L n, \infty)^d} \exp \left[ - d\,z_+ - e^{- z_+} \right] \dd \zz
\mbox{\quad ($u_j \equiv e^{- z_j}$)} \\
&= \frac{1}{(d - 1)!} \int_{- \L n}^{\infty}\!(\L n + x)^{d - 1} \exp\left( - d\,x - e^{-x} \right) \dd x
\mbox{\quad ($x = z_+$)} \\
&= \frac{1}{(d - 1)!} \int_0^n\!(\L n - \L y)^{d - 1} y^{d - 1} e^{-y} \dd y
\mbox{\quad ($x = - \L y$)} \\
&= \frac{(\L n)^{d - 1}}{(d - 1)!} \sum_{j = 0}^{d - 1} {{d - 1} \choose j} \frac{(-1)^j}{(\L n)^j}
\int_0^n\!(\L y)^j y^{d - 1} e^{-y} \dd y \\
&= (\L n)^{d - 1} \sum_{j = 0}^{d - 1} \frac{(-1)^j}{j! (d - 1 - j)!} (\L n)^{-j} \int_0^n\!(\L y)^j y^{d - 1} e^{-y} \dd y.\\
\end{align*}
But
\begin{align*}
 \int_0^n\!(\L y)^j y^{d - 1} e^{-y} \dd y
&=  \int_0^{\infty}\!(\L y)^j y^{d - 1} e^{-y} \dd y - \int_n^{\infty}\!(\L y)^j y^{d - 1} e^{-y} \dd y \\
&=  \Gamma^{(j)}(d) - \Theta((\L n)^j n^{d - 1} e^{-n}).
\end{align*}
The desired result now follows easily.
\end{proof}

We next bound the difference between $\hat{\iota}_{d, n}$ and
\begin{equation}
\label{Idntilde}
\tilde{\iota}_{d, n}
:= n^d \int_{[0, 1)^d} x_{\times}^{d - 1} (1 - x_{\times})^n \dd \xx.
\end{equation}

\begin{lemma}
\label{L:quadratic}
For fixed $d \geq 1$, as $n \to \infty$ we have
\[
0 \leq \hat{\iota}_{d, n} - \tilde{\iota}_{d, n} = O(n^{-1} (\L n)^{d - 1}).
\]
\end{lemma}

\begin{proof}
We utilize the elementary inequality
\[
e^{- n t} (1 - n t^2) \leq (1 - t)^n \leq e^{- n t}
\]
for $n \geq 1$ and $0 \leq t \leq 1$ (see~\cite[Lemma~5]{Bai(2001)}).  This yields
\[
0 \leq \hat{\iota}_{d, n} - \tilde{\iota}_{d, n}
\leq n^{d + 1} \int_{[0, 1)^d} x_{\times}^{d + 1} \exp(- n x_{\times}) \dd \xx.
\]
Proceeding just as in the proof of \refL{L:Poissonized}, we find that the last expression here is
$O(n^{-1} (\L n)^{d - 1})$.
\end{proof}

\begin{theorem}
\label{T:Idn}
For fixed $d \geq 1$, as $n \to \infty$ the expected number of interior generators at time~$n$ in dimension~$d$ satisfies
\[
\E \iota_{d, n}
= (\L n)^{d - 1} \sum_{j = 0}^{d - 1} \frac{(-1)^j \Gamma^{(j)}(d)}{j! (d - 1 - j)!} (\L n)^{-j} + O(n^{-1} (\L n)^{d - 1}).
\]
\end{theorem}

\begin{proof}
Comparing~\eqref{Idn} and~\eqref{Idntilde} and then invoking \refL{L:quadratic}, we see that
\begin{align*}
\E \iota_{d, n}
&= \frac{n\fall{d}}{(n - d)^d} \tilde{\iota}_{d, n - d}
= [1 + O(n^{-1})]\,\tilde{\iota}_{d, n - d} \\
&= [1 + O(n^{-1})]\,\left[ \hat{\iota}_{d, n - d} + O(n^{-1} (\L n)^{d - 1}) \right] \\
&= [1 + O(n^{-1})]\,\hat{\iota}_{d, n - d} + O(n^{-1} (\L n)^{d - 1}).
\end{align*}
But, according to \refL{L:Poissonized},
\begin{align*}
\hat{\iota}_{d, n - d}
&= [\L (n - d)]^{d - 1} \sum_{j = 0}^{d - 1} \frac{(-1)^j \Gamma^{(j)}(d)}{j! (d - 1 - j)!} [\L (n - d)]^{-j}
+ O((n \L n)^{d - 1} e^{-n}) \\
& = (\L n)^{d - 1} \sum_{j = 0}^{d - 1} \frac{(-1)^j \Gamma^{(j)}(d)}{j! (d - 1 - j)!} (\L n)^{-j}
+ O(n^{-1} (\L n)^{d - 2}).
\end{align*}
Thus
\begin{align*}
\E \iota_{d, n}
&= [1 + O(n^{-1})]\,(\L n)^{d - 1} \sum_{j = 0}^{d - 1} \frac{(-1)^j \Gamma^{(j)}(d)}{j! (d - 1 - j)!} (\L n)^{-j}
+ O(n^{-1} (\L n)^{d - 1}) \\
&= (\L n)^{d - 1} \sum_{j = 0}^{d - 1} \frac{(-1)^j \Gamma^{(j)}(d)}{j! (d - 1 - j)!} (\L n)^{-j}
+ O(n^{-1} (\L n)^{d - 1}),
\end{align*}
as claimed.
\end{proof}

Combining~\eqref{Gdn} and~\eqref{Idn}, we can obtain an exact expression for $\E \gamma_{d, n}$.  Similarly, combining~\eqref{Gdn} and \refT{T:Idn} we obtain the following asymptotic expansion in powers of logarithm for $\E \gamma_{d, n}$ after a little rearrangement.

\begin{theorem}
\label{T:Gdnasymptotics}
For fixed $d \geq 1$, as $n \to \infty$ the expected number of generators at time~$n$ in dimension~$d$ satisfies
\[
\E \gamma_{d, n}
= (\L n)^{d - 1} \sum_{j = 0}^{d - 1} a_{d, j} (\L n)^{-j} + O(n^{-1} (\L n)^{d - 1}),
\]
where
\[
a_{d, j}
:= \sum_{k = 0}^j {d \choose {d - j + k}} \frac{(-1)^{k} \Gamma^{(k)}(d - j + k)}{k! (d - 1 - j)!}.
\nopf
\]
\end{theorem}

\begin{remark}
\label{R:time}
(a)~In particular, $a_{d, 0} = 1$, so $\E \gamma_{d, n}$ has lead-order asymptotics
\[
\E \gamma_{d, n}
= (\L n)^{d - 1} + O((\L n)^{d - 2});
\]
this is $(d - 1)!$ times as large as the lead-order asymptotics for the expected number of remaining records, namely,
\[
\E \rho_{d, n} = \frac{(\L n)^{d - 1}}{(d - 1)!} + O((\L n)^{d - 2}).
\]

(b)~For $d = 2$ and $n \geq 0$, we have
\[
\E \gam_{2, n} = H_n + 1 = \E \rho_{2, n} + 1,
\]
where $H_n := \sum_{k = 1}^n k^{-1}$ is the $n$th harmonic number; and in fact it is easy to see that 
$\gam_{2, n} = \rho_{2, n} + 1$.  For $d = 3$ and $n \geq 0$, we have
\[
\E \gam_{3, n} = H_n^2  + H^{(2)}_n + 1 = 2 \E \rho_{3, n} + 1,
\]
where $H^{(2)}_n := \sum_{k = 1}^n k^{-2}$ is the $n$th second-order harmonic number; and in fact 
$\gam_{3, n} = 2 \rho_n + 1$, as established in \cite[Corollary~6.6]{Fillgenerating(2018)}.  There is not such a simple relationship between the exact values of $\rho_{d, n}$ and $\gam_{d, n}$ for $d \geq 4$; confer
 \cite[Remark~6.7]{Fillgenerating(2018)}.  

(c)~We hope to extend the work of this section by finding at least lead-order asymptotics for the variance, and also a normal approximation or other limit theorem, for the number $\gamma_{d, n}$ of generators after~$n$ observations.
\end{remark}

\section{Stochastic lower bound on $F^-_n$ via the first moment method}
\label{S:LB-}

In this section we show how to obtain a suitable stochastic lower bound on $F_n^-$.  See \refP{P:-lower} for the result.  The idea, for a suitably chosen sequence $(b_n)$ is to apply the first moment method (computation of sufficiently small mean, together with application of Markov's inequality) to the count
\[
\gamma_n(b) := (\mbox{number of generators at epoch~$n$ with coordinate-sum $\leq b$}).
\]
The bound on the mean of $\gamma_n(b_n)$ is obtained by 
suitably modifying the proof of \refT{T:Gdnasymptotics} [compare also the similar treatment of $\rho_n(b_n)$ in \refS{S:UBh-}].

\begin{lemma}
\label{L:UB mean hat}
With the notation and assumptions of~\eqref{bndefhat}--\eqref{gbdef} and~\eqref{Jjdef}, as $n \to \infty$ we have
\begin{equation}
\E \gam_n(b_n) 
\leq (1 + o(1)) \frac{(\L n)^{d - 1}}{(d - 1)!} (c_n \L_2 n)^{d - 1} (\L n)^{ - c_n}.
\end{equation}
\end{lemma}

\begin{proof}
We will be very brief here.  Following very closely along the lines of \refS{S:expected},
one finds that
\begin{align*}
\E \gamma_n(b_n) 
&\sim \frac{1}{(d - 1)!} \int_{n e^{-b_n}}^{\L n} (\L n - \L z)^{d - 1} z^{d - 1} e^{-z} \dd z \\
&\leq \frac{(\L n)^{d - 1}}{(d - 1)!} \int_{n e^{-b_n}}^{\infty} z^{d - 1} e^{-z} \dd z \\
&\sim \frac{(\L n)^{d - 1}}{(d - 1)!} (n e^{-b_n})^{d - 1} \exp(- n e^{-b_n}) \\
&= \frac{(\L n)^{d - 1}}{(d - 1)!} (c_n \L_2 n)^{d - 1} (\L n)^{ - c_n}.
\end{align*}
\end{proof}

We are now in position to utilize Markov's inequality.

\begin{proposition}[\underline{Stochastic lower bound on $F_n^-$}]
\label{P:-lower}
Fix $d \geq 2$.  If $1 \leq c_n = O(1)$ and
\[
b \equiv b_n := \L n - \L_3 n - \L c_n,
\]
then
\[
\P(F^-_n \leq b_n) 
\leq \E \gamma_n(b_n) 
\leq (1 + o(1)) \frac{(\L n)^{d - 1}}{(d - 1)!} (c_n \L_2 n)^{d - 1} (\L n)^{ - c_n}.
\]
\end{proposition}

\section{Proof of \refT{T:-}}
\label{S:proof-}
In this section we prove \refT{T:-}.

\begin{proof}[Proof of \refT{T:-}]
\ \vspace{-.1in}\\

(a)~This follows readily from Propositions~\ref{P:-lower} and \ref{P:-upper} (or one can invoke \refT{T:hF-} instead of \refP{P:-upper}).

(b)~As noted in \refR{R:compare}, this is immediate from \refT{T:hF-}(b), already established in \refS{S:proofh-}. 

(c)~This follows in the same fashion as our given proof of \refT{T:hF-}(c), now using \refP{P:-lower} in place of \refP{P:h-lower}.  We leave the routine details to the reader.
\ignore{
If
$(b_n)$ is (ultimately) monotone nondecreasing and $(n_j)$ is any strictly increasing sequence of positive integers, then
\[
\{F_n^- \leq b_n\mbox{\rm\ \io$(n)$}\} \subseteq \{F_{n_j}^- \leq b_{n_{j + 1}}\mbox{\rm\ \io$(j)$}\}.
\]
To complete the proof of part~(c), we choose $b_n \equiv \L n - \L_3 n - \L d - \L 2 - c$ with 
$c > 0$ and $n_j \equiv 2^j$, bound $\P(F_{n_j}^- \leq b_{n_{j + 1}})$ using \refP{P:-lower}, and apply the first Borel--Cantelli lemma.

Here are the details.  First note that
\[
b_{2 n} = \L(2 n) - \L_3(2 n) - \L d - \L 2 - c \leq \L n - \L_3 n - \L d - c, 
\] 
the bounding expression being the one in \refP{P:-lower} with $c_n \equiv e^c d$.  Thus, by that proposition,
\begin{align*}
\P(F_{n_j}^- \leq b_{n_{j + 1}})
&\leq \P(F_{n_j}^- \leq \L n_j - \L_3 n_j - \L d - c) \\
&= O((\L n_j)^{- [e^c d - (d - 1)]} (\L_2 n_j)^{d - 1}) \\
&= O(j^{- [1 + (e^c d - 1)]} (\L j)^{d - 1}), 
\end{align*}
which is summable.
}
\end{proof}

\appendix

\ccreset

\section{Proof of \refL{L:mean hat}}\label{A:meanproof}

This appendix is devoted to the (elementary) proof of \refL{L:mean hat}.

\begin{proof}[Proof of \refL{L:mean hat}]
We will prove~\eqref{meanhat} 
by separately considering (a)~upper and (b)~lower bounds.
Before beginning, we note that the mean in question has the exact expression
\begin{align}
\E \rho_n(b_n) 
&= n \int_{x \geq 0:\,x_+ \leq b_n}\!e^{- x_+} (1 - e^{- x_+})^{n - 1} \dd x \nonumber \\
&= \frac{n}{(d - 1)!} \int_0^{b_n}\!y^{d - 1} e^{-y} (1 - e^{-y})^{n - 1} \dd y. \label{exact mean}
\end{align}
A key technical tool we will use is the pair of elementary inequalities
\begin{equation}
\label{L5}
e^{- n t} (1 - n t^2) \leq (1 - t)^n \leq e^{- n t}
\end{equation}
for $n \geq 1$ and $0 \leq t \leq 1$ (see~\cite[Lemma~5]{Bai(2001)}).
Also, note from the definition~\eqref{Jjdef} of the function $J_j$ that
\begin{equation}
\label{Jjasy}
J_j(x) \sim (\L x)^j e^{-x}\mbox{\ as $x \to \infty$} 
\end{equation}
and that for $1 \leq x < y$ we have
\begin{equation}
\label{Jjdiff}
0 < J_j(x) - J_j(y) \leq (\L y)^j (e^{-x} - e^{-y}) = (\L y)^j e^{-y} (e^{y - x} - 1).
\end{equation}
\smallskip

(a)~Utilizing the upper bound in~\eqref{L5} immediately we derive
\begin{align}
\lefteqn{\E \rho_{n + 1}(b_{n + 1})} \nonumber \\
&= \frac{n + 1}{(d - 1)!} \int_0^{b_{n + 1}}\!y^{d - 1} e^{-y} (1 - e^{-y})^n \dd y \nonumber \\
&\leq \frac{n + 1}{(d - 1)!} \int_0^{b_{n + 1}}\!y^{d - 1} \exp\left( - n e^{-y} - y \right) \dd y \label{expression} \\
&= \frac{1 + n^{-1}}{(d - 1)!} \int_{n e^{-b_{n + 1}}}^n\!(\L n - \L z)^{d - 1} e^{-z} \dd z \nonumber \\
&\leq \frac{1 + n^{-1}}{(d - 1)!} \int_{n e^{-b_{n + 1}}}^{\infty}\![\L (n + 1) - \L z]^{d - 1} e^{-z} \dd z \nonumber \\
&= \frac{1 + n^{-1}}{(d - 1)!} \sum_{j = 0}^{d - 1} (-1)^j \binom{d - 1}{j} [\L (n + 1)]^{d - 1 - j} 
\int_{n e^{-b_{n + 1}}}^{\infty}\!(\L z)^j e^{-z} \dd z \nonumber \\
&\leq \frac{1 + n^{-1}}{(d - 1)!} \sum_{j = 0}^{d - 1} (-1)^j \binom{d - 1}{j} [\L(n + 1)]^{d - 1 - j}
J_j(n e^{ - b_{n + 1}}). \nonumber
\end{align}
That is,
\begin{equation}
\label{Erhonbnbound}
\E \rho_n(b_n)
\leq \frac{1 + (n - 1)^{-1}}{(d - 1)!} \sum_{j = 0}^{d - 1} (-1)^j \binom{d - 1}{j} (\L n)^{d - 1 - j}
J_j((n - 1) e^{ - b_n}). 
\end{equation}
By~\eqref{Jjdiff} we have
\begin{align}
0
&< J_j((n - 1) e^{-b_n}) - J_j(e^{\gb_n}) \nonumber \\ 
&\leq (\L \gb_n)^j e^{-\gb_n} [\exp(e^{ - b_n}) - 1] \nonumber \\
&\sim (\L \gb_n)^j e^{-\gb_n} e^{ - b_n} \label{Jdiffconsequence} \\
&= (\L \gb_n)^j e^{-\gb_n} n^{-1} c_n \L_2 n \nonumber \\
&= O\!\left( (\L \gb_n)^j e^{-\gb_n} n^{-1} \L_2 n \right). \nonumber
\end{align}
Thus
\begin{align*}
\lefteqn{\E \rho_n(b_n)} \\
&\leq \frac{1}{(d - 1)!} \sum_{j = 0}^{d - 1} (-1)^j \binom{d - 1}{j} (\L n)^{d - 1 - j}
J_j(\gb_n) + O(n^{-1} (\L n)^{d - 1 - c_n} \L_2 n) \\
&= [1 + O(n^{-1} \L_2 n)] \frac{1}{(d - 1)!} \sum_{j = 0}^{d - 1} (-1)^j \binom{d - 1}{j} (\L n)^{d - 1 - j} J_j(\gb_n),
\end{align*}
bettering the claim in the upper-bound direction for the mean at~\eqref{meanhat}.
\smallskip

(b)~Utilizing the lower bound in~\eqref{L5}, we find from~\eqref{exact mean} that
\begin{align}
\E \rho_n(b_n) 
&= \frac{n}{(d - 1)!} \int_0^{b_n}\!y^{d - 1} e^{-y} (1 - e^{-y})^{n - 1} \dd y \nonumber \\
&\geq \frac{n}{(d - 1)!} \int_0^{b_n}\!y^{d - 1} e^{-y} (1 - e^{-y})^n \dd y \nonumber \\
&\geq \frac{n}{(d - 1)!} \int_0^{b_n}\!y^{d - 1} \exp\left( - n e^{-y} - y \right) \left( 1 - n e^{ - 2 y} \right) \dd y. \label{mean LB}
\end{align}
We derive that the added term in~\eqref{mean LB} satisfies 
\begin{align*}
\lefteqn{\frac{n}{(d - 1)!} \int_0^{b_n}\!y^{d - 1} \exp\left( - n e^{-y} - y \right) \dd y} \\
&= \frac{1}{(d - 1)!} \int_{\gb_n}^n\!(\L n - \L z)^{d - 1} e^{-z} \dd z \\
&= \frac{1}{(d - 1)!} \sum_{j = 0}^{d - 1} (-1)^j \binom{d - 1}{j} (\L n)^{d - 1 - j} 
[J_j(\gb_n) - J_j(n)].
\end{align*}
But $J_j(n) \sim e^{-n} (\L n)^j$, whence it follows simply that the added term in~\eqref{mean LB} is lower-bounded by
\[
\frac{1}{(d - 1)!} \sum_{j = 0}^{d - 1} (-1)^j \binom{d - 1}{j} (\L n)^{d - 1 - j} J_j(\gb_n) - O(e^{ - n} (\L n)^{d - 1}).
\]
 
So it remains to show that the subtracted term in~\eqref{mean LB} can be absorbed into the remainder term in~\eqref{meanhatalt}, which we will do in similar (but easier) fashion to upper-bounding $\E \rho_n(b_n)$.
Indeed, the subtracted term satisfies
\begin{align}
0
&< \frac{n^2}{(d - 1)!} \int_0^{b_n}\!y^{d - 1} \exp\left( - n e^{-y} - 3 y \right) \dd y \label{subtracted} \\
&= \frac{n^{-1}}{(d - 1)!} \int_{\gb_n}^n\!z^2 (\L n - \L z)^{d - 1} e^{-z} \dd z \nonumber \\
&\leq \frac{n^{-1} (\L n)^{d - 1}}{(d - 1)!}\int_{\gb_n}^{\infty}\!z^2 e^{-z} \dd z \nonumber \\
&\sim \frac{n^{-1} (\L n)^{d - 1}}{(d - 1)!} \gb_n^2 e^{-\gb_n} \nonumber \\
&= \frac{n^{-1} (\L n)^{d - 1}}{(d - 1)!} c_n^2 (\L_2 n)^2 (\L n)^{ - c_n} \nonumber \\ 
&= O(n^{-1} (\L_2 n)^2 (\L n)^{d - 1 - c_n}) \nonumber \\
&= O\!\left( n^{-1} (\L_2 n)^2
\frac{1}{(d - 1)!} \sum_{j = 0}^{d - 1} (-1)^j \binom{d - 1}{j} (\L n)^{d - 1 - j} J_j(\gb_n) \right), \nonumber
\end{align}
as desired.
\end{proof}

\section{Proof of \refL{L:UB variance}}\label{A:varproof}

This appendix is devoted mainly to the proof of \refL{L:UB variance}.  Following that, in \refR{R:varlower} we establish the reverse asymptotic inequality.

\begin{proof}[Proof of \refL{L:UB variance}]
For convenience (to match up with the treatment of $\Var \rho_n$ in \cite[Section~2]{Bai(2005)}), at various points in the proof (including, notably, at the start) we switch from \iid\ Exponential$(1)$ observation coordinates to observations uniformly distributed in the cube $[0, 1]^d$ and record-\emph{small} values.  

We start with the second factorial moment of $\rho_{n + 2}(b_{n + 2})$. Let $B^{(i)}_n$ denote the event that $X^{(i)}$ is a remaining record at epoch~$n$.  Then
\begin{align}
M_{n + 2}
&:= \E\left[ \rho_{n + 2}(b_{n + 2})^2 \right] - \E \rho_{n + 2}(b_{n + 2}) \nonumber \\
&= (n + 2) (n + 1) 
\P\!\left( B^{(1)}_{n + 2} \cap B^{(2)}_{n + 2} \cap \bigcap_{i = 1}^2 \left\{ X^{(i)}_{\times} \geq e^{-b_{n + 2}} \right\} \right) \nonumber \\ 
&= (n + 2) (n + 1) \int^*\!\left[ 1 - x_{\times} - y_{\times} + (x \wedge y)_{\times} \right]^n \dd x \dd y \label{Mn+2}\\
&\leq (n + 2) (n + 1) \int^*\!\exp\left[- n (x_{\times} + y_{\times} - (x \wedge y)_{\times}) \right] \dd x \dd y \nonumber \\
&= K_n + I_n, \nonumber
\end{align}
with
\begin{equation}
\label{Kn}
K_n 
:= (n + 2) (n + 1) 
\int^*\!\exp\left[- n (x_{\times} + y_{\times}) \right] \dd x \dd y
\end{equation}
and
\begin{equation}
\label{In}
I_n := (n + 2) (n + 1) \int^*\!\exp\left[- n (x_{\times} + y_{\times}) \right] 
\left( \exp\left[ n (x \wedge y)_{\times} \right] - 1 \right) \dd x \dd y,
\end{equation}
where the integrals $\int^*$ are over $x, y \in [0, 1]^d$ incomparable in the dominance ordering and satisfying $x_{\times} \geq e^{-b_{n + 2}}$ and $y_\times \geq e^{-b_{n + 2}}$.

Continuing to follow the outline in \cite[Section~2]{Bai(2005)} for calculating $\Var \rho_n$, we observe
that $K_n$ is bounded above by the same expression without the incomparability restriction, which (switching back to Exponential coordinates) in turn equals $(n + 2) / (n + 1)$ times the square of
\begin{align}
\lefteqn{\hspace{-.5in}(n + 1) \int_{x \in [0, 1]^d:\,x_{\times} \geq e^{- b_{n + 2}}}\!\exp(- n x_{\times}) \dd x} 
\nonumber \\
&= (n + 1) \int_{x \geq 0:\,x_+ \leq b_{n + 2}}\!e^{- x_+} \exp(- n e^{- x_+}) \dd x \nonumber \\
&= \frac{n + 1}{(d - 1)!} \int_0^{b_{n + 2}}\!y^{d - 1} \exp\left( - n e^{-y} - y \right) \dd y; \label{expression2} 
\end{align}
this last expression~\eqref{expression2} agrees with~\eqref{expression} except that $b_{n + 1}$ there is replaced by $b_{n + 2}$.  By the same argument as follows after~\eqref{expression} in \refApp{A:meanproof}, \eqref{expression2} can be bounded above by
\begin{align}
\lefteqn{\hspace{-1in}[1 + O(n^{-1} \L_2 n)] 
\frac{1}{(d - 1)!} \sum_{j = 0}^{d - 1} (-1)^j \binom{d - 1}{j} [\L (n + 2)]^{d - 1 - j} J_j(\gb_{n + 2})} \nonumber \\
&= \left[1 + O\!\left( n^{-1} (\L_2 n)^2 \right) \right] \E \rho_{n + 2}(b_{n + 2}), \label{expression2bound}
\end{align}
where the equality here follows from \refL{L:mean hat}.  We conclude that $K_n$ is bounded above by
\begin{align*}
\lefteqn{\hspace{-.5in}\left[1 + O\!\left( n^{-1} (\L_2 n)^2 \right) \right] [\E \rho_{n + 2}(b_{n + 2})]^2} \\ 
&= [\E \rho_{n + 2}(b_{n + 2})]^2 + O(n^{-1} (\L n)^{2 (d - 1 - c_n)} (\L_2 n)^2). 
\end{align*}
Therefore, we have
\begin{align}
\Var \rho_{n + 2}(b_{n + 2}) 
&= \E \rho_{n + 2}(b_{n + 2}) + M_{n + 2} - [\E \rho_{n + 2}(b_{n + 2})]^2 \nonumber \\
&\leq \E \rho_{n + 2}(b_{n + 2}) + I_n + O(n^{-1} (\L n)^{2 (d - 1 - c_n)} (\L_2 n)^2) \nonumber \\
&= (1 + o(1)) \E \rho_{n + 2}(b_{n + 2}) + I_n. \label{varboundIn}
\end{align}
We remark in passing that one can see from~\eqref{expression2bound} and~\eqref{varboundIn} why we needed such detailed (lower-bound) asymptotics for $\E \rho_n(b_n)$ in order to get~\eqref{varboundIn}: We needed for the big-oh remainder term on the right in~\eqref{expression2bound} to have the property that its product with $\E \rho_n(b_n) = \Theta((\L n)^{d - 1 - c_n})$ be $o(1)$. 

We turn our attention now to $I_n = \sum_{k = 1}^{d - 1} \binom{d}{k} I_{n, k}$; here 
\begin{equation}
\label{Inkdef}
I_{n, k} := (n + 2) (n + 1) \int^{**}\!\exp[ - n u_{\times} (\Pi' v_i + \Pi'' v_i)]
\left( \exp\left[ n\,u_{\times} v_{\times} \right] - 1 \right) u_{\times} \dd \uu \dd \vv,
\end{equation}
where the integral 
$\int^{**}$
is over $\uu = (u_1, \ldots, u_d) \in [0, 1]^d$ and $\vv = (v_1, \ldots, v_d) \in [0, 1]^d$ satisfying $u_{\times} \Pi' v_i \geq e^{-b_{n + 2}}$ and $u_{\times} \Pi'' v_i \geq e^{-b_{n + 2}}$, with 
$\Pi' v_i := \prod_{i = 1}^k v_i$ and $\Pi' v_i := \prod_{i = k + 1}^d v_i$.  Continuing to follow the outline in \cite[Section~2]{Bai(2005)}, 
with $\tgb_n := n e^{ - b_{n + 2}} = \frac{n}{n + 2} \gb_{n + 2}$ and
$\int^{***}$ denoting integration over $(x, z, y)$ satisfying $x \in (0, 1)$, 
$z \in (0, 1)$, $y \in (0, n)$, $y x \geq \tgb_n$, and $y z \geq \tgb_n$, 
we find
\begin{align}
\lefteqn{[(1 + 2 n^{-1}) (1 + n^{-1})]^{-1} I_{n, k}} \nonumber \\
&= \frac{1}{(d - 1)!} \label{Inkalt} \\ 
&{} \times \int^{***}\!(\L n - \L y)^{d - 1} \frac{( - \L x)^{k - 1}}{(k - 1)!} 
\frac{( - \L z)^{d - 1 - k}}{(d - 1 - k)!} y e^{ - y (x + z)} (e^{y x z} - 1) \dd x \dd z \dd y. \nonumber
\end{align}
We can lower-bound $\L y$ in the first integrand factor by $\L \tgb_n = \Theta(\L_3 n)$
and thus
\begin{align}
\lefteqn{[(1 + 2 n^{-1}) (1 + n^{-1})]^{-1} I_{n, k}} \nonumber \\
&\leq \frac{(\L n)^{d - 1}}{(d - 1)! (k - 1)! (d - 1 - k)!} \nonumber \\ 
&{} \hspace{0.5in} \times \int_0^1\!\int_0^1\! [ - \L (x \wedge z)]^{d - 2}
\int\!y \left[ e^{ - y (x + z - x z)} - e^{ - y (x + z)} \right] \dd y \dd z \dd x \nonumber \\
&= \frac{2 (\L n)^{d - 1}}{(d - 1)! (k - 1)! (d - 1 - k)!} \nonumber \\ 
&{} \hspace{0.5in} \times \int_0^1\!\int_x^1\! ( - \L x)^{d - 2}
\int_{\tgb_n / x}^{\infty}\!y \left[ e^{ - y (x + z - x z)} - e^{ - y (x + z)} \right] \dd y \dd z \dd x; \label{Inkbound}
\end{align}
in the penultimate expression the $y$-integral is over $y \in (0, \infty)$ satisfying $y x > \tgb_n$ and 
$y z > \tgb_n$, and in the last expression the $y$-integral equals
\begin{align}
\lefteqn{(x + z - x z)^{-2} 
[1 + (x + z - x z) x^{-1} \tgb_n] \exp[- (x + z - x z) x^{-1} \tgb_n]} \nonumber \\
&{} \qquad {} - (x + z)^{-2} [1 + (x + z) x^{-1} \tgb_n] \exp[- (x + z) x^{-1} \tgb_n]. \label{zint} 
\end{align}
The integral over $z \in (x, 1)$ of~\eqref{zint} equals
\begin{align*}
D_n(x) 
&:= x^{-1} (1 - x)^{-1} (2 - x)^{-1} e^{ - (2 - x) \tgb_n} - \half x^{-1} e^{ - 2 \tgb_n} \\
&{} \qquad \qquad {} - (1 - x)^{-1} e^{ - x^{-1} \tgb_n} + (1 + x)^{-1} e^{- (x^{-1} + 1) \tgb_n}. 
\end{align*}
\ignore{
Therefore,
\begin{align*}
\lefteqn{I_{n, k}(b)} \\
&\leq 2 \frac{(\L n)^{d - 1}}{(d - 1)!} \int_0^1\!\int_x^1\! 
\frac{( - \L x)^{k - 1}}{(k - 1)!} \frac{( - \L z)^{d - 1 - k}}{(d - 1 - k)!} \\
&{} \times \left\{
(x + z - x z)^{-2} [1 + (x + z - x z) x^{-1} n e^{-b}] \exp[- (x + z - x z) x^{-1} n e^{-b}] \right. \\
&{} \qquad \left.{} - (x + z)^{-2} [1 + (x + z) x^{-1} n e^{-b}] \exp[- (x + z) x^{-1} n e^{-b}] \right\} \dd z \dd x 
\end{align*}
Bounding $( - \L z)^{d - 1 - k}$ from above by $( - \L x)^{d - 1 - k}$ in this last expression, we find 
\begin{align*}
\lefteqn{I_{n, k}(b)} \\
&\leq 2 \frac{(\L n)^{d - 1}}{(d - 1)!} \int_0^1\! 
\frac{( - \L x)^{d - 2}}{(k - 1)! (d - 1 - k)!} \\
&{} \times \int_x^1\!
\left\{ (x + z - x z)^{-2} [1 + (x + z - x z) x^{-1} n e^{-b}] \exp[- (x + z - x z) x^{-1} n e^{-b}] \right. \\
&{} \qquad \left.{} - (x + z)^{-2} [1 + (x + z) x^{-1} n e^{-b}] \exp[- (x + z) x^{-1} n e^{-b}] \right\} \dd z \dd x \\
&= 2 \frac{(\L n)^{d - 1}}{(d - 1)!} \int_0^1\! 
\frac{( - \L x)^{d - 2}}{(k - 1)! (d - 1 - k)!} \\
&{} \qquad \times
\left\{ x^{-1} (1 - x)^{-1} (2 - x)^{-1} e^{ - (2 - x) n e^{-b}} - \half x^{-1} e^{ - 2 n e^{-b}} \right. \\ 
&{} \qquad \qquad 
\left.{} - (1 - x)^{-1} e^{ - x^{-1} n e^{-b}} + (1 + x)^{-1} e^{- (x^{-1} + 1) n e^{-b}} \right\} \dd x.  
\end{align*}
}

For $0 < x \leq 1/2$ we use the bound
\[
D_n(x) \leq x^{-1} e^{ - (2 - x) \tgb_n} [(1 - x)^{-1} (2 - x)^{-1} - \half e^{ - x \tgb_n}] + e^{- 3 \tgb_n}.
\]
Noting
\begin{align*}
(1 - x)^{-1} (2 - x)^{-1} - \half e^{ - x \tgb_n}
&\leq (1 - x)^{-1} (2 - x)^{-1} - \half (1 - x \tgb_n) \\
&=  \half x [(1 - x)^{-1} (2 - x)^{-1} (3 - x) + \tgb_n], 
\end{align*}
we conclude
\begin{align*}
D_n(x) 
&\leq \half e^{ - (2 - x) \tgb_n} [\tgb_n + (1 - x)^{-1} (2 - x)^{-1} (3 - x)] + e^{ - 3 \tgb_n} \\
&\leq \half e^{ - (2 - x) \tgb_n} (\tgb_n + \tfrac{10}{3}) + e^{ - 3 \tgb_n} \\
&\leq (1 + o(1)) \half \tgb_n e^{ - (2 - x) \tgb_n} \\ 
&\leq (1 + o(1)) \half \tgb_n e^{ - \tfrac32 \tgb_n} 
\end{align*}
uniformly for $0 < x \leq 1/2$.

For $1/2 \leq x < 1$ we use the bound
\[
D_n(x) \leq (1 - x)^{-1} e^{ - (2 - x) \tgb_n} [x^{-1} (2 - x)^{-1} - e^{ - x^{-1} (1 - x)^2 \tgb_n}] 
+ \tfrac23 e^{- 2 \tgb_n}.
\]
Noting
\begin{align*}
x^{-1} (2 - x)^{-1} - e^{ - x^{-1} (1 - x)^2 \tgb_n}
&\leq x^{-1} (2 - x)^{-1} - [1 -  x^{-1} (1 - x)^2 \tgb_n] \\
&= x^{-1} (1 - x)^2 [(2 - x)^{-1} + \tgb_n],
\end{align*}
we conclude
\begin{align*}
D_n(x)
&\leq x^{-1} (1 - x) e^{ - (2 - x) \tgb_n} [\tgb_n + (2 - x)^{-1}] + \tfrac23 e^{- 2 \tgb_n} \\
&\leq 2 (1 - x) e^{ - (2 - x) \tgb_n} (\tgb_n + 1) + \tfrac23 e^{- 2 \tgb_n} \\
&\leq (1 + o(1))\,2 (1 - x) \tgb_n e^{ - (2 - x) \tgb_n} + \tfrac23 e^{ - 2 \tgb_n}
\end{align*} 
uniformly for $1/2 \leq x < 1$.

Therefore, for fixed~$d$ and~$k$ we have
\[
(\L n)^{- (d - 1)} I_{n, k} = O\!\left( \tgb_n e^{ - \tfrac32 \tgb_n} \right) + O(\tgb_n J_{n, k}), 
\]
where
\[
J_{n, k} := \int_{1/2}^1\!( - \L x)^{d - 2} (1 - x) e^{ - (2 - x) \tgb_n} \dd x.
\]
For any given $\eps > 0$, we can choose $\gd > 0$ so that 
$- \L x \leq (1 + \eps) (1 - x)$ for $x \in [1 - \gd, 1]$; note that $\gd$ does not depend on~$n$.  Then
\begin{align*}
J_{n, k}
&\leq \tfrac18 (\L 2)^{d - 2} e^{ - (1 + \gd) \tgb_n} 
+ (1 + \eps)^{d - 2} e^{ - \tgb_n} \int_{1 - \gd}^1\!(1 - x)^{d - 1} e^{ - (1 - x) \tgb_n} \dd x \\ 
&\leq \tfrac18 (\L 2)^{d - 2} e^{ - (1 + \gd) \tgb_n} 
+ (1 + \eps)^{d - 2} e^{ - \tgb_n} \int_0^\infty\!y^{d - 1} e^{ - y \tgb_n} \dd y \\
&= \tfrac18 (\L 2)^{d - 2} e^{ - (1 + \gd) \tgb_n} 
+ (1 + \eps)^{d - 2} (d - 1)! \tgb_n^{-d} e^{ - \tgb_n}. 
\end{align*}
It follows that
\[
J_{n, k} \leq (1 + o(1))\,(d - 1)! \tgb_n^{-d} e^{ - \tgb_n}
\]
and hence that
\begin{align}
I_{n, k} 
&= O\!\left( (\L n)^{d - 1} \tgb_n e^{ - \tfrac32 \tgb_n} \right) + O((\L n)^{d - 1} \tgb_n J_{n, k}) \nonumber \\ 
&= O\!\left( (\L n)^{d - 1} \tgb_n^{ - (d - 1)} e^{ - \tgb_n} \right)
= O\!\left( (\L_2 n)^{-2} [\L(n + 2)]^{d - 1 - c_{n + 2}} \right) \nonumber \\
&= O\!\left( (\L_2 n)^{-2} \E \rho_{n + 2}(b_{n + 2}) \right) \nonumber \\
&= o(\E \rho_{n + 2}(b_{n + 2})). \label{Ink}  
\end{align}

Combining~\eqref{varboundIn} with~\eqref{Ink}, we finally find that
\begin{align*}
\Var \rho_n(b_n)
&\leq (1 + o(1))\,\E \rho_n(b_n),
\end{align*}
as claimed.
\end{proof}

\begin{remark}
\label{R:varlower}
Application of the second-moment method in \refS{S:UBh-} requires only the upper bound on 
$\Var \rho_n(b_n)$ in \refL{L:UB variance}.  However, we find it interesting that, with the notation and assumptions of~\eqref{bndefhat}, as $n \to \infty$ the reverse asymptotic inequality
\begin{equation}
\label{varLB}
\Var \rho_n(b_n) \geq (1 + o(1)) \E \rho_n(b_n)
\end{equation}
also holds.  This suggests a Poisson approximation for the distribution of $\rho_n(b_n)$; see \refR{R:rholimit} for further comments.

\begin{proof}[Proof of~\eqref{varLB}]
Recall from~\eqref{Mn+2} that
\begin{align*}
M_{n + 2}
&= \E [\rho_{n + 2}(b_{n + 2})^2] - \E \rho_{n + 2}(b_{n + 2}) \\
&= (n + 2) (n + 1) \int^*\!\left[ 1 - x_{\times} - y_{\times} + (x \wedge y)_{\times} \right]^n \dd x \dd y. 
\end{align*}
From~\eqref{L5} we then have
\[
M_{n + 2} 
\geq K_n + I_n - \tK_n - \tI_n
\geq K_n - \tK_n - \tI_n;
\]
here $K_n$ and $I_n$ are defined at \eqref{Kn}--\eqref{In}, and
\begin{align}
\tK_n 
&:= (n + 2) (n + 1) n
\int^*\!\left[ x_{\times} + y_{\times} - (x \wedge y)_{\times} \right]^2
\exp\left[- n (x_{\times} + y_{\times}) \right] \dd x \dd y \nonumber \\
&\leq (n + 2) (n + 1) n
\int^*\!\left( x_{\times} + y_{\times} \right)^2
\exp\left[- n (x_{\times} + y_{\times}) \right] \dd x \dd y \nonumber \\
&= 2 (n + 2) (n + 1) n
\int^*\!x_{\times}^2
\exp\left[- n (x_{\times} + y_{\times}) \right] \dd x \dd y \label{tKnbound}
\end{align}
and (similarly)
\begin{align}
\tI_n 
&:= (n + 2) (n + 1) n \nonumber \\ 
&{} \hspace{-.1in}\times \int^*\!\left[ x_{\times} + y_{\times} - (x \wedge y)_{\times} \right]^2
\exp\left[- n (x_{\times} + y_{\times}) \right] 
\left( \exp\left[ n (x \wedge y)_{\times} \right] - 1 \right) \dd x \dd y \nonumber \\
&\leq 2 (n + 2) (n + 1) n \int^*\!x_{\times}^2
\exp\left[- n (x_{\times} + y_{\times}) \right] 
\left( \exp\left[ n (x \wedge y)_{\times} \right] - 1 \right) \dd x \dd y, \label{tInbound}
\end{align}
where once again the integrals $\int^*$ are over $x, y \in [0, 1]^d$ incomparable in the dominance ordering and satisfying $x_{\times} \geq e^{-b_{n + 2}}$ and $y_\times \geq e^{-b_{n + 2}}$.  To complete the proof of~\eqref{varLB} we will show that
\begin{align}
\tK_n &= o(\E \rho_{n + 2}(b_{n + 2})), \label{tKno} \\
\tI_n &= o(\E \rho_{n + 2}(b_{n + 2})), \label{tIno} \\
K_n &\geq [\E \rho_{n + 2}(b_{n + 2})]^2 - o(\E \rho_{n + 2}(b_{n + 2})). \label{Knlowerbound}
\end{align}

To prove~\eqref{tKno}, we observe that $\tK_n$ is bounded above by the same expression as~\eqref{tKnbound} without the incomparability restriction; thus (switching back to Exponential coordinates),
\begin{align}
\tK_n
&\leq 2 (n + 2) (n + 1) n \int_{x \in [0, 1]^d:\,x_{\times} \geq e^{- b_{n + 2}}}\,
x_{\times}^2 \exp(- n x_{\times}) \dd x \nonumber \\
&{} \qquad \qquad \qquad \qquad \qquad \times
\int_{y \in [0, 1]^d:\,x_{\times} \geq e^{- b_{n + 2}}}\,\exp(- n y_{\times}) \dd y \nonumber \\
&= 2 (n + 2) n \int_{x \geq 0:\,x_+ \leq b_{n + 2}}\!e^{- 3 x_+} \exp(- n e^{- x_+}) \dd x \nonumber \\
&{} \qquad \qquad \times
(n + 1) \int_{x \geq 0:\,x_+ \leq b_{n + 2}}\!e^{- x_+} \exp(- n e^{- x_+}) \dd x \nonumber \\
&= \frac{2 (n + 2) n}{(d - 1)!} \int_0^{b_{n + 2}}\!y^{d - 1} \exp\left( - n e^{-y} - 3 y \right) \dd y \label{line1}\\
&{} \qquad \qquad \times
\frac{n + 1}{(d - 1)!} \int_0^{b_{n + 2}}\!y^{d - 1} \exp\left( - n e^{-y} - y \right) \dd y. \label{line2}
\end{align}
The factor~\eqref{line2} agrees with~\eqref{expression2} and so by now-familiar arguments is asymptotically equivalent to 
$\E \rho_{n + 2}(b_{n + 2})$, and the factor~\eqref{line1} equals $2 (n + 2) / n$ times the expression at~\eqref{subtracted}, except that $b_n$ there is replaced by $b_{n + 2}$, and so by arguments that are also now familiar the first-line factor is 
$O\!\left( n^{-1} (\L_2 n)^2 (\L n)^{d - 1 - c_n} \right) = o(1)$.  We have established~\eqref{tKno}.

To prove~\eqref{tIno}, we again drop the incomparability restriction and proceed for $\tI_n$ as we did for $I_n$ in the proof of \refL{L:UB variance}.  This leads to $\tI_n = \sum_{k = 1}^{d - 1} \tI_{n, k}$, where [with notation just as in the proof of \refL{L:UB variance}, compare~\eqref{Inkdef}--\eqref{Inkbound}]
\begin{align*}
\lefteqn{[(1 + 2 n^{-1}) (1 + n^{-1})]^{-1} \tI_{n, k}} \\ 
&:= \frac{2 n^3}{(d - 1)!} \\ 
&{} \qquad \times \int^{**}\!\exp[ - n u_{\times} (\Pi' v_i + \Pi'' v_i)]
\left( \exp\left[ n\,u_{\times} v_{\times} \right] - 1 \right) u_{\times}^3 (\Pi' v_i)^2 \dd \uu \dd \vv \\
&= \frac{2 n^{-1}}{(d - 1)!} \\ 
&{} \times \int^{***}\!(\L n - \L y)^{d - 1} \frac{( - \L x)^{k - 1}}{(k - 1)!} 
\frac{( - \L z)^{d - 1 - k}}{(d - 1 - k)!} x^2 y^3 e^{ - y (x + z)} (e^{y x z} - 1) \dd x \dd z \dd y \\
&\leq \frac{4 n^{-1} (\L n)^{d - 1}}{(d - 1)! (k - 1)! (d - 1 - k)!} \\ 
&{} \hspace{0.5in} \times \int_0^1\!\int_x^1\! x^2 ( - \L x)^{d - 2}
\int_{\tgb_n / x}^{\infty}\!y^3 \left[ e^{ - y (x + z - x z)} - e^{ - y (x + z)} \right] \dd y \dd z \dd x \\
& = O\!\left( n^{-1} (\L n)^{d - 1} \right) = o(\E \rho_{n + 2}(b_{n + 2})),
\end{align*}
because the last triple integral, with the lower limit of integration $\tgb_n / x$ on~$y$ replaced by~$0$, is easily shown to converge.  [We could even use the dominated convergence theorem to conclude that $\tI_n = o\!\left( n^{-1} (\L n)^{d - 1} \right)$.]  We have established~\eqref{tIno}.

It remains to prove~\eqref{Knlowerbound}.  For that we write $K_n$ of \eqref{Kn} as the difference 
$\hK_n - (1 + 2 n^{-1}) (1 + n^{-1}) I_{n, 0}$, where $\hK_n$ is obtained from~\eqref{Kn} by extending the region of integration to \emph{all} $x, y \in [0, 1]^d$, and 
\begin{equation}
\label{In0def}
I_{n, 0} := 2 n^2 \int\!\exp[ - n (x_{\times} + y_{\times})] \dd x \dd y,
\end{equation}
where the integral in~\eqref{In0def} is over $x, y \in [0, 1]^d$ satisfying $y < x$ and 
$y_{\times} \geq e^{ - b_{n + 2}}$.  We will prove that 
(i) $\hK_n \geq [\E \rho_{n + 2}(b_{n + 2})]^2 - o(\E \rho_{n + 2}(b_{n + 2}))$ and 
(ii) $I_{n, 0} = o(\E \rho_{n + 2}(b_{n + 2}))$; that will establish~\eqref{Knlowerbound} and complete the proof of~\eqref{varLB}.

The proof of~(i) is easy.  Indeed, it follows using~\eqref{L5} that, as desired,
\[
\hK_n 
\geq \frac{n + 1}{n + 2} [\E \rho_{n + 2}(b_{n + 2})]^2 
= [\E \rho_{n + 2}(b_{n + 2})]^2 - o(\E \rho_{n + 2}(b_{n + 2})). 
\]

Finally, we prove~(ii).  By now-familiar changes of variables, we find
\begin{align*}
I_{n, 0} 
&= \frac{2}{[(d - 1)!]^2} \int_0^n\!\int_0^1\!{\bf 1}(y z \geq \tgb_n) 
(\L n - \L y)^{d - 1} ( - \L z)^{d - 1} y e^{- y (1 + z)} \dd z \dd y \\
&\leq \frac{2 (\L n)^{d - 1}e^{ - \tgb_n}}{[(d - 1)!]^2} 
\int_0^{\infty}\!\int_0^1\!{\bf 1}(y z \geq \tgb_n) ( - \L z)^{d - 1} y e^{- y} \dd z \dd y,
\end{align*}
where we remind the reader that $\tgb_n$ is defined as in the proof of \refL{L:UB variance} as $\tgb_n = n e^{ - b_{n + 2}} = \frac{n}{n + 2} \beta_{n + 2}$, and hence 
\[
e^{ - \tgb_n} 
= e^{ - \gb_{n + 2}} e^{ - 2 \gb_{n + 2} / (n + 2)} 
= (1 + o(1)) e^{ - \gb_{n + 2}}.
\]
Therefore,
\begin{equation}
\label{In0bound}
I_{n, 0}
\leq (1 + o(1)) \frac{2 \E \rho_{n + 2}(b_{n + 2})}{(d - 1)!} 
\int_0^{\infty}\!\int_0^1\!{\bf 1}(y z \geq \tgb_n) ( - \L z)^{d - 1} y e^{- y} \dd z \dd y. 
\end{equation}
But
\[
\int_0^{\infty}\!\int_0^1\!( - \L z)^{d - 1} y e^{- y} \dd z \dd y = \Gamma(d) < \infty, 
\]
so [since $\tgb_n \sim \gb_n = c_n \L_2 n = \Theta(\L_2 n) \to \infty$] it follows from the dominated convergence theorem that the double integral in~\eqref{In0bound} tends to~$0$ as $n \to \infty$, and hence that $I_{n, 0} = o(\E \rho_{n + 2}(b_{n + 2}))$, as claimed.
\end{proof}
\end{remark}
%

\bibliography{records}
\bibliographystyle{plain}

\end{document}